\documentclass[reqno]{amsart}
\usepackage{amssymb, amsthm, amsmath, amsfonts, amscd, mathrsfs}
\usepackage{times}
\usepackage[breaklinks,colorlinks,citecolor=teal,linkcolor=teal,urlcolor=teal,hyperindex]{hyperref}
\usepackage{color}
\usepackage{tikz}
\usetikzlibrary{arrows} \usetikzlibrary{patterns}
\usepackage[all,cmtip]{xy} \input diagxy

\usepackage[top=2.5cm,bottom=2.5cm,left=2.6cm,right=2.6cm]{geometry}

\newtheorem{theorem}{Theorem}[section]
\newtheorem{lemma}[theorem]{Lemma}
\newtheorem{prop}[theorem]{Proposition}
\newtheorem{defn-prop}[theorem]{Definition-Proposition}

\newtheorem{coro}[theorem]{Corollary}

\newtheorem{defn}[theorem]{Definition}

\theoremstyle{definition}
\newtheorem{example}[theorem]{Example}

\newtheorem{remark}[theorem]{Remark}

\numberwithin{equation}{section}

\def \msf{\mathsf}

\def \mc{\mathcal}
\def \scr{\mathscr}
\def \inv{^{-1}}
\def \0{\infty}
\def \v{\vskip 0.1in}
\def \n{\noindent}

\def \integer{\mathbb{Z}}

\def \qq{\quad}

\def \rto{\rightarrow}
\def \Rto{\Rightarrow}
\def \hrto{\hookrightarrow}
\def \rrto{\rightrightarrows}

\def \A{{\sf A}}
\def \B{{\sf B}}
\def \C{{\sf C}}

\def \G{{\sf G}}
\def \sH{{\sf H}}

\def \K{{\sf K}}

\def \U{{\sf U}}
\def \V{{\sf V}}

\def \Z{{\sf Z}}

\def \f{{\sf f}}
\def \g{{\sf g}}
\def \h{{\sf h}}
\def \su{{\sf u}}
\def \sv{{\sf v}}
\def \w{{\sf w}}

\def \sq{{\sf q}}
\def \sk{{\sf k}}
\def \sj{{\sf j}}

\allowdisplaybreaks

\begin{document}

\title{On fibrations of Lie groupoids}

\author{Bohui Chen}
\address{Department of Mathematics and Yangtze Center of Mathematics, Sichuan University, 610065, Chengdu, China.}
\email{bohui@cs.wisc.edu.}

\author{Cheng-Yong Du}
\address{Department of Mathematics, Sichuan Normal University, 610068, Chengdu, China.}
\email{cyd9966@hotmail.com}

\author{Yu Wang}
\address{School of Mathematics, Sichuan University of Science $\&$ Engineering, 643000, Zigong, China}
\email{wangyu$\underline{\hspace{0.5em}}$813@163.com}

\date{}

\thanks{}


\begin{abstract}
As groupoids generalize groups, motivated by group extensions we consider a kind of fibrations of Lie groupoids, called locally topological product Lie groupoid fibrations with fiber $\A$, i.e.,
\[
1\rto \A\rto\G\rto\K\rto 1
\]
where $\A,\G$ and $\K$ are Lie groupoids.  Similar to the theory of group extensions, we show that the existence of locally topological product Lie groupoid fibrations with fiber $\A$ over $\K$ is obstructed by a groupoid cohomology of $H^3_{\bar \Lambda}(\K,Z_\A)$, and these locally topological product Lie groupoid fibrations are classified by $H^2_{\bar \Lambda}(\K,Z_\A)$ once exists. Here $Z_\A$ is the center of  $\A$. This generalizes the theory of group extensions, of gerbes over manifolds/groupoids and etc.
\end{abstract}

\maketitle


\section{Introduction}
\label{intro}

Groupoid fibrations were introduced by R. Brown (cf. \cite{Brown1970}) in 1970. A groupoid fibration is a star-surjective strict morphism of groupoids $\f=(f^0,f^1):\G\rto \K$. Given such a groupoid fibration, the fiber of $\f$ over an object $y\in K^0$ is a sub-groupoid of $\G$, that is $\left((f^1)\inv(1_y)\rrto (f^0)\inv(y)\right)$, where $1_y$ is the unit arrow over $y$. Naturally, we could consider groupoid fibrations with a fixed fiber. A natural example of groupoid fibrations with a fixed fiber comes out of group extensions. Consider a group extension of $K$ by $A$
\begin{equation}\label{eqn_gext}
1\rto A\xrightarrow{i} G\xrightarrow{j} K\rto 1.
\end{equation}
Associated to each group we have the classifying groupoid; so for groups in the extension \eqref{eqn_gext} we have calssifying groupoids $[\bullet/G]=(G\rrto \bullet)$, $[\bullet/K]=(K\rrto\bullet)$ and $[\bullet/A]=(A\rrto \bullet)$. Then from \eqref{eqn_gext} we get a groupoid fibration $[\bullet/G]\xrightarrow{(id_\bullet,j)} [\bullet/K]$ with fiber $[\bullet/A]$.

It is well-known that a group extension \eqref{eqn_gext} is determined by a pair $(\omega, f)$, where $\omega: K\rto \mathrm{Aut}(A)$ which induces a homomorphism $\bar\omega:K\rto \mathrm{Out}(A)$, called the band, $f: K\times K\rto A$ is called a cofactor; moreover $(\omega,f)$ satisfies certain cocycle condition (see for example \cite{Adem-Milgram2004,Weibel2002}). Here $\mathrm{Out}(A)$ denotes the group of outer automorphism of $A$. Conversely, given a homomorphism $\bar\omega: K\to \mathrm{Out}(A)$, it defines an obstruction class $c_{\bar\omega}\in H^3_{\bar\omega}(K,Z(A))$, where $Z(A)$ is the center of $A$. Extensions of $K$ by $A$ with $\bar\omega$ as its band exist
if and only if $c_{\bar\omega}=0$. Moreover, once $c_{\bar\omega}=0$, extensions with band $\bar\omega$ are classified by $H^2_{\bar\omega}(K,Z(A))$ (see for example \cite{Adem-Milgram2004,Weibel2002}).

A natural question is to study the existence and classification of groupoid fibrations over $\K$ with a given groupoid as fibers.
In this paper we study this question in the category of Lie groupoids. We consider the generalizations of \eqref{eqn_gext}
\begin{equation}\label{eqn_gpext}
1\rto \sf A\xrightarrow{\phi} \sf G\xrightarrow{\psi} K\rto 1,
\end{equation}
where $\A,\G,\K$ are Lie groupoids. Here by the exactness at $\sf A$ we mean that $\phi$ is injective, and by the exactness at $\sf K$ we mean that $\psi$ is surjective. The exactness at $\sf G$ will be defined in Section \ref{S 3}. In fact, there are many important examples that can be thought as special cases of this concept: (1) when $\A$ is $[\bullet/A]$ for some group $A$, $\G$ is a gerbe over $\K$ (see for example \cite{Laurent-Gengoux-Stienon-Xu2009,Tang-Tseng2014a, Tu2006}); (2) when $\A$ is a manifold $F$, $\G$ is a fiber bundle over $\K$ with fiber $F$ (see for example  \cite{Adem-Leida-Ruan2007}).

The main results of the paper show that the theory of group extensions may be generalized to that of certain Lie groupoid fibrations, which we call locally topological product Lie groupoid fibrations (cf. Definition \ref{D locally topological product fibration}). In Section \ref{S 4} we show that the existence of locally topological product Lie groupoid fibration with fiber $\A$ associated to a given morphism
\[
\bar\Lambda:K^1\rto\overline{\sf SAut(A)},
\]
called the band,
is obstructed by a class $[\Xi_{\bar\Lambda}]$ in the groupoid cohomology group $H^3_{\bar\Lambda}(\K,Z_\A)$, where $\overline{\sf SAut(A)}$ is the coarse space of the groupoid of strict automorphisms of $\A$ and also a group, and $Z_\A$ is the center of the Lie groupoid $\A$ (cf. Definition \ref{D center-of-groupoid}). In Section \ref{S 5}, we show that when $[\Xi_{\bar\Lambda}]=0$, the isomorphic classes of such locally topological product Lie groupoid fibrations with band $\bar\Lambda$ are bijective to $H^2_{\bar\Lambda}(\K,Z_\A)$. We also classify locally topological product Lie groupoid fibration under certain equivalence relation.

The study of Lie groupoid fibrations is originally motivated by the study of Gromov--Witten theory of orbifolds (\cite{Chen-Ruan2002}) in terms of orbifold groupoids (i.e. proper \'etale Lie groupoids). This is due to an alternative view of \eqref{eqn_gpext}: roughly speaking, $\sf G$ may be thought as a family of $\sf A$ parameterized by $\sf K$. In the Gromov--Witten theory of orbifolds, it is inevitable to consider families of curves parameterized by the Deligne--Mumford moduli spaces which are Lie groupoids (cf. \cite{Chen-Li-Wang2016}).
This issue will be addressed in \cite{Chen-Du-Ono}.

Another interesting case is to consider the symplectic orbifold groupoid fibrations. Suppose that we have a locally topological product Lie groupoid fibration given by \eqref{eqn_gpext}, one may ask when will $\G$ be a symplectic orbifold if both $\sf A$ and $\sf K$ are symplectic orbifolds. We explain this issue in Section \ref{S 6}. Furthermore, it is natural to ask the relations among the Gromov--Witten invariants of $\sf A,K$ and $\sf G$.
For instance, Tang and Tseng (\cite{Tang-Tseng2014a,Tang-Tseng2016}) studied some special cases of that $\sf G$ is an \'etale gerbe over a symplectic orbifold $\sf K$. We will study the Gromov--Witten theory for general fibrations in the future.
Finally, as an application, we study the structure of certain morphism groupoids of morphisms to a locally topological product Lie groupoid fibrations in Section \ref{S 7}.

In \cite{Buss-Meyer2016}, A. Buss and R. Meyer introduced and studied topological groupoid fibrations. It would be interesting to study locally topological product Lie groupoid fibrations in the frame work of topological groupoid fibrations. In this paper, we focus on Lie groupoids.

\subsection*{Acknowledgements}
We would like to thank Rui Wang and Xiang Tang for useful discussions. This work was supported by the National Natural Science Foundation of China (No. 11431001, No. 11890663, No. 11501393, No. 11501390, No. 11701397), by Sichuan Science and Technology Program (No. 2019YJ0509, No. 2019YJ0456), by Sichuan University (No. 1082204112190).

\section{Preliminaries}\label{S 2}

\subsection{Lie groupoids}
We briefly review basic concepts of Lie groupoids. One may be referred  to \cite{Adem-Leida-Ruan2007,Mackenzie1987,Mackenzie2005,Moerdijk-Mrcun2003} for more details.

A Lie groupoid, denoted by $\G=(G^1\rrto G^0)$, consists of two smooth manifolds $G^0$ and $G^1$, and structure maps:
\begin{enumerate}
\item  the source and target maps by $s: G^1\rto G^0$ and $t: G^1\rto G^0$ respectively,
\item  the composition/multiplication map\footnote{Note that this is different from the usual notation, in which $s(g)=t(h)$, that is the arrows go from the right to the left. In this paper we use the convention that the composition of arrows goes from left to right.} by $m:G^{[2]}:=G^1\times_{t,s}G^1\rto G^1$, $m(g,h)=gh$ or $g\cdot h$,
\item  the unit map by $u: G^0\rto G^1$, and $u(a)$ by $1_a$,
\item  the inverse map $i: G^1\rto G^1$, and $i(g)$ by $g\inv$,
\end{enumerate}
where both source and target maps are submersions, and all other three maps are smooth maps.

A Lie groupoid is a category after forgetting the smooth structures. We denote an arrow in $g\in G^1$ by $g: x\rto y$. We denote the {\bf coarse space} of $\G$ by $\overline{\G}=G^0/G^1$, which has the quotient topology, and the quotient map to the coarse space by $\pi:G^0\rto\overline\G$.

A Lie groupoid $\G$ is {\bf proper} if $s\times t: G^1\rto G^0\times G^0$ is proper; $\G$ is {\bf \'etale} if both $s$ and $t$ are local diffeomorphisms. A proper \'etale Lie groupoid is called an {\bf orbifold groupoid}.

Let $\G,\sH$ be two Lie groupoids. A {\bf strict Lie groupoid morphism} (or strict morphism for short) $\f:\G\rto \sH$ is a functor $\f=(f^0,f^1)$ with both $f^0$ and $f^1$ are smooth maps. A strict Lie groupoid morphism $\f:\G\rto\sH$ induces a continuous map on coarse space $\overline{\f}:\overline{\G}\rto\overline{\sH}$.

A {\bf natural transformation} $\sigma:\f\Rto\g:\G\rto\sH$ between two strict morphisms is a smooth map $\sigma:G^0\rto H^1$ such that for every arrow $g:x\rto y$ in $G^1$ there is a commutative diagram in $H^1$:
\begin{align}\label{E commut-diag-natural-trans}
\begin{split}
\xymatrix{
f^0(x)\ar[rr]^-{\sigma(x)} \ar[d]_-{f^1(g)}&& g^0(x)\ar[d]^-{g^1(g)}\\
f^0(y)\ar[rr]^-{\sigma(y)} && g^0(y).}
\end{split}
\end{align}

A strict Lie groupoid morphism $\f:\G\rto\sH$ is called an {\bf equivalence} if
\begin{enumerate}
\item $t\circ pr_2:G^0\times_{f^0,s}H^1\rto H^0$ is a surjective submersion, and
\item the square
\begin{align*}
\begin{split}
\xymatrix{
G^1\ar[rr]^-{f^1}\ar[d]_-{s\times t} &&
H^1 \ar[d]^-{s\times t} \\
G^0\times G^0\ar[rr]^-{f^0\times f^0} && H^0 }
\end{split}
\end{align*}
is a fiber product.
\end{enumerate}

\begin{remark}\label{R 2-cat-of-grpd}
Lie groupoids together with strict morphisms and natural transformations between strict morphisms form a 2-category {\bf 2Gpd}. We denote the vertical and horizontal compositions of natural transformations in {\bf 2Gpd} by ``$\odot$'' and ``$\circledast$'' respectively. Therefore, for three natural transformations $\rho_1:\f\Rto\g:\A\rto\B, \rho_2:\g\Rto \h:\A\rto\B$ and $\rho_3:\sk\Rto\sj:\B\rto\C$ we have $\rho_1\odot\rho_2:\f\Rto\h:\A\rto \B$ with
\begin{align}\label{E def-odot}
\rho_1\odot\rho_2(a)=\rho_1(a)\cdot\rho_2(a), \qq \forall\, a\in A^0,
\end{align}
and $\rho_3\circledast\rho_1:\sk\circ\f\Rto\sj\circ\g:\A\rto\C$ with
\begin{align}\label{E def-circledast}
\rho_3\circledast\rho_1(a)=(k^1\circ\rho_1(a))\cdot (\rho_3\circ g^0(a)), \qq \forall \, a\in A^0.
\end{align}
We denote by $1_\f:\f\Rto\f$ the identity natural transformations, hence $1_\f(x)=1_{f^0(x)}$.

Moreover, the category of strict morphisms $\sf SMor(\A,\B)=(\mathrm{SMor}^1(\A,\B)\rrto \mathrm{SMor}^0(\A,\B))$ is a groupoid with multiplication over $\mathrm{SMor}^1(\A,\B)$ being the vertical composition of natural transformation.
\end{remark}

\subsection{Refinements of Lie groupoid via open covers}
Let $\G=(G^1\rrto G^0)$ be a Lie groupoid. Let $\mc U=\{U_a\mid a\in\scr A\}$ be an open cover of $G^0$. Then we can form the {\bf refinement groupoid} (or {\bf pullback groupoid}) $\G[\mc U]$ as follows:
\begin{enumerate}
\item the object space is
\begin{align*}
\G[\mc U]^0:=\bigsqcup_{a\in\mc A} U_a,
\end{align*}
the disjoint union of open subsets in $\mc U$; there is a natural inclusion map $q^0_{\mc U}:\G[\mc U]^0\rto G^0$ that embeds each $U_a$ into $G^0$;
\item the arrow space is
\begin{align*}\begin{split}
\G[\mc U]^1&=\G[\mc U]^0\times_{q^0_{\mc U},s}G^1\times_{t,q^0_{\mc U}}\G[\mc U]^0
=\{(x,g,y)\in \G[\mc U]^0\times G^1\times \G[\mc U]^0 \mid g:q^0_{\mc U}(x)\rto q^0_{\mc U}(y)\};
\end{split}\end{align*}
there is also a natural map $q^1_{\mc U}:\G[\mc U]^1\rto G^1$ that projects to the second factor.
\end{enumerate}
With the obvious five structure maps we see that $\G[\mc U]:=(\G[\mc U]^1\rrto \G[\mc U]^0)$ is a Lie groupoid. Moreover
\begin{align*}
\sq_{\mc U}:=(q^0_{\mc U},q^1_{\mc U}): \G[\mc U]\rto\G
\end{align*}
is an equivalence of Lie groupoids.

If there is another open cover $\mc W=\{W_b\mid b\in \scr B\}$ that refines $\mc U$ via an refine map $\iota_{\mc W,\mc U}:\mc W\rto \mc U$, then it is direct to see that $\iota_{\mc W,\mc U}$ induces a Lie groupoid equivalence $\iota_{\mc W,\mc U}:\G[\mc W]\rto \G[\mc U]$. Moreover we have the following commutative diagram of Lie groupoid equivalences
\begin{align*}
\begin{split}
\xymatrix{
\G[\mc W]\ar[rr]^-{\iota_{\mc W,\mc U}} \ar[drr]_-{\sq_{\mc W}} &&
\G[\mc U]\ar[d]^-{\sq_{\mc U}} \\
&&\K.}
\end{split}
\end{align*}

Suppose $\phi:\G\rto\K$ is a strict morphism and $\mc U$ is an open cover of $K^0$. $\phi$ pulls back $\mc U$ to an open cover $\phi^*\mc U$ of $G^0$. Then we have a pullback strict morphism $\sq_\mc U^*\phi:\G[\phi^*\mc U]\rto\K[\mc U]$ and a commutative diagram of strict morphisms
\begin{align*}
\begin{split}
\xymatrix{
\G[\phi^*\mc U]\ar[rr]^-{\sq_{\phi^*\mc U} } \ar[d]_-{\sq_{\mc U}^*\phi} && \G\ar[d]^-\phi \\
\K[\mc U]\ar[rr]^-{\sq_{\mc U}} &&\K.}
\end{split}
\end{align*}

\subsection{Center of Lie groupoids}
We recall the  center for a groupoid that is defined in
\cite[\S 5.1]{Chen-Du-WangR2019}. Here we only consider Lie groupoids. Let $\A$ be a Lie groupoid. For every $x\in A^0$, let $\Gamma_x:=\{g\in A^1\mid g:x\rto x\}$ be the isotropy group of $x$. Let $ZA^0:=\{g\in Z(\Gamma_x) \mid x\in A^0\}\subseteq A^1$ where $Z(\Gamma_x)$ means the center of the group $\Gamma_x$. There is an $\A$-action\footnote{See for example \cite{Chen-Hu2006} for the definition of groupoid actions on spaces.} on $ZA^0$, of which the {\em anchor map} is the natural projection
\begin{align*}
p=s=t:ZA^0\rto A^0,\qq Z(\Gamma_x)\ni g\mapsto x
\end{align*}
and the {\em action map} is
\begin{align*}
\mu: A^1\times_{s,p}ZA^0\rto ZA^0,\qq (h,g)\mapsto h\inv g h.
\end{align*}
The center groupoid $\Z\A$ is the action groupoid
\begin{align}\label{E def-center-groupoid-ZA}
\Z\A:=\A\ltimes ZA^0=(A^1\times_{s,p}ZA^0\rrto ZA^0),
\end{align}
whose source map is projection to the second factor and target map is $\mu$. Other structure maps are obvious. We have a natural projection $\msf p=(p^0,p^1):\Z\A\rto\A$ with
\begin{align*}
p^0=p:ZA^0\rto A^0,\qq p^1=pr_1:A^1\times_{s,p}ZA^0\rto A^1.
\end{align*}

\begin{defn} \label{D center-of-groupoid}
The {\bf center $Z_\A$} of $\A$ is the set of global sections of $\sf p:\Z\A\rto \A$.
\end{defn}
It is clear that $Z_\A$ is an abelian group.

\begin{remark}\label{R smoothness-of-ZA}
Note that $ZA^0$ is not a (disjoint union of) smooth submanifold of $A^1$ in general, it can be very singular. However, we can also talk about the smoothness of sections of $\sf p: ZA\rto A$. For instance, we call a section $\sigma$ smooth if $\sigma:A^0\rto ZA^0\subseteq A^1$ is a smooth map when we view it as a map $A^0\rto A^1$. $Z_\A$ only contains smooth sections of $\sf p:ZA\rto A$ in this sense.

Let $M$ be a smooth manifold, then we say that a map $f:M\rto Z_\A$ is smooth if $M\times A^0\rto ZA^0\hrto A^1$ is smooth.
\end{remark}
\subsection{Automorphism groupoid of a Lie groupoid}

In \cite{Chen-Du-WangR2019} we studied the automorphism groupoid of a given groupoid. Here, we consider a simpler version which is used in this paper.

Let $\A=(A^1\rrto A^0)$ be a Lie groupoid. Denote the identity morphism of $\A$ by $\sf id_A$. A strict Lie groupoid morphism $\f:\A\rto \A$ is called an {\bf strict automorphism} of $\A$ if there is another strict Lie groupoid morphism $\g:\A\rto\A$ such that $\f\circ\g=\sf id_\A$ and $\g\circ\f=\sf id_\A$.

Let $\mathrm{SAut}^0(\A)$ be the set of strict automorphisms of $\A$. With composition of functors as the multiplication, it becomes a group. Let $\mathrm{SAut}^1(\A)$ be the space of natural transformations between strict automorphisms of $\A$. Then we have a groupoid of strict automorphisms
\begin{align*}
\sf SAut(A)=(\mathrm{SAut}^1(\A)\rrto \mathrm{SAut}^0(\A)).
\end{align*}
The structure maps are obvious, and the multiplication in $\mathrm{SAut}^1(\A)$ is the vertical composition ``$\odot$'' between natural transformations\footnote{In fact, $\sf SAut(A)$ is a subgroupoid of $\sf SMor(A,A)$.}. Let $s_{saut}, t_{saut}: \mathrm{SAut}^1(\A)\rto \mathrm{SAut}^0(\A)$ denote the source and target maps. Set $s_{saut}^{-1}(\sf id_A)$ to be $N_\A$.

\begin{prop} \label{P SAut}
We have the following facts.
\begin{enumerate}
\item $N_\A$ is a group with respect to the horizontal composition $\circledast$ in {\bf 2Gpd}.
\item $\sf SAut(A)$ is a quotient groupoid $\mathrm{SAut}^0(\A)
\rtimes N_\A$.
\item The stabilization of $N_\A$ action at $\sf id_\A$ is
$Z_\A$. So is the stabilization of other automorphisms in $\mathrm{SAut}^0(\A)$.
\end{enumerate}
\end{prop}
\begin{proof} (1)
we could write down composition $\circledast$ explicitly. First note that we have
\begin{equation}\label{eqn_na}
N_\A=\left\{\sigma\in C^\infty(A^0,A^1)\left|\substack{s\circ \sigma=id_{A^0}, \text{ and }\\ t\circ \sigma:A^0\rto A^0\text{ is a diffeomorphism}}\right.\right\}.
\end{equation}
 Let $\sigma,\gamma\in N_\A$, we have (comparing with \eqref{E def-circledast})
\begin{align*}
(\gamma\circledast\sigma)(x)=\sigma(x)\cdot \gamma(t(\sigma(x))),
\qq
\text{i.e.}
\qq
\xymatrix{
x\ar[r]^-{\sigma(x)} &
y \ar[r]^-{\gamma(y)} &z.}
\end{align*}
The inverse of $\sigma$ is given by $\sigma^{\circledast,-1}(x)=\sigma((t\circ\sigma)\inv(x))\inv$. That is, if $\sigma:\sf id_A\Rto\f$, then $\sigma^{\circledast,-1}:\sf id_A\Rto\f\inv$, and $\sigma\circledast \sigma^{\circledast,-1}=\sigma^{\circledast,-1}\circledast\sigma =1_{\sf id_A}:\sf id_A\Rto id_A$.

(2) First of all, we write down the action of $N_\A$ on $\mathrm{SAut}^0(\A)$. Take an $\alpha:{\sf id_A\Rto f} \in N_\A$ and $\sf g\in \mathrm{SAut}^0(\A)$. Then we set
\[
\alpha(\g):= t_{saut}(\alpha)\circ\g=\f\circ \g.
\]
This is an $N_\A$-action since
\[
(\beta\circledast\alpha)(\g)=
t_{saut}(\beta\circledast\alpha)\circ\g=
t_{saut}(\beta)\circ t_{saut}(\alpha)\circ \g=
\beta(\alpha(\g)).
\]

We next give an isomorphism $\phi:{\sf SAut(A)}\rto \mathrm{SAut}^0(\A)\rtimes N_\A$. We set \[
\phi^0=id_{\mathrm{SAut}^0(\A)},\qq \text{and}\qq
\phi^1(\alpha:\f\Rto\g)=(\f, \alpha\circledast 1_{\f\inv}:\sf id_A\Rto\g\circ\f\inv).
\]
The inverse of $\phi$ is
\[
\psi^0=id_{\mathrm{SAut}^0(\A)},\qq\text{and}\qq
\psi^1(\f,\alpha:\sf id_\A\Rto g)=\alpha\circledast 1_{\f}:\f\Rto \g\circ\f.
\]
We next show that $\psi$ is a morphism. We compute
\[
\psi^1\Big((\f\circ \g,\beta:\sf id_\A\Rto h)\circledast(\g,\alpha:\sf id_\A\Rto f)\Big)=\psi^1(\g,\beta\circledast\alpha:\sf id_A\Rto h\circ f)=\beta\circledast\alpha\circledast 1_\g
\]
and
\begin{align*}
&\psi^1(\f\circ \g,\beta:\sf id_\A\Rto h)\odot\psi^1(\g,\alpha:\sf id_\A\Rto f) \\
& =(\beta\circledast 1_{\f\circ\g}) \odot(\alpha\circledast 1_\g)=(\beta\circledast 1_{\f\circ\g}) \odot(1_{\sf id_A}\circledast(\alpha\circledast 1_\g))\\
&=(\beta\odot 1_{\sf id})\circledast (1_{\f\circ\g}\odot(\alpha\circledast 1_\g))
=\beta\circledast\alpha\circledast 1_\g,
\end{align*}
where we have used the fact that $\circledast$ is a composition functor in the 2-category {\bf 2Gpd}. Therefore $\psi$ is a morphism. Similarly $\phi$ is also a morphism. So ${\sf SAut(A)}\cong \mathrm{SAut}^0\rtimes N_\A$ as a groupoid.

(3)
Suppose that  $\sigma\in N_\A$ satisfies $\sigma:\sf id_\A\Rto id_\A$,
then by the commutative diagram \eqref{E commut-diag-natural-trans} we get $\sigma$ is a section of
$\sf p: ZA\rto A$, hence $\sigma\in Z_\A$. It follows from the definition of $N_\A$-action, the stabilization of every automorphism in $\mathrm{SAut}^0(\A)$ is the same as the one of $\sf id_A$. This finishes the proof.
\end{proof}

\begin{lemma}
The image $t_{saut}(N_\A)$ is a normal subgroup of $\mathrm{SAut}^0(\A)$.
\end{lemma}
\begin{proof}
Suppose $\sigma:\sf id_A\Rto f$, i.e. $\f\in t_{saut}(N_\A)$, and $\h\in\mathrm{SAut}^0(\A)$. We next show that $\h\circ\f\circ\h\inv\in t_{saut}(N_\A)$.  Let $1_\h:\h\Rto\h$ and $1_{\h\inv}:\h\inv\Rto\h\inv$ be the identity natural transformations over $\h$ and $\h\inv$ respectively. Then
\begin{align*}
1_\h\circledast \sigma \circledast 1_{\h\inv}:\sf h\circ id_A \circ h\inv\Rto \h\circ\f\circ\h\inv.
\end{align*}
It follows from $\sf h\circ id_A \circ h\inv=id_A$ that $
\h\circ\f\circ\h\inv=t_{saut}(1_\h\circledast \sigma \circledast 1_{\h\inv})$. This finishes the proof.
\end{proof}
To summarize, we have an exact sequence of groups
\begin{align}\label{E SES-1}
\xymatrix{
1\ar[r] & Z_\A\ar@{^(->}[r] & N_\A \ar[r]^-{t_{saut}} &\mathrm{SAut}^0(\A) \ar[r]^-\pi &\mathrm{SAut}^0(\A)/\text{Im}\,t_{saut}\ar[r] &1.
}
\end{align}

\begin{coro}
The coarse space $\overline{\sf{SAut}(\A)}$ is the group
isomorphic to
$ \mathrm{SAut}^0(\A)/\mathrm{Im}\,t_{saut}. $
\end{coro}

Therefore the exact sequence \eqref{E SES-1} can also be written as
\begin{align}\label{E SES-2}
\xymatrix{
1\ar[r] & Z_\A\ar@{^(->}[r] & N_\A \ar[r]^-{t_{saut}} &\mathrm{SAut}^0(\A) \ar[r]^\pi &\overline{\sf{SAut}(\A)}\ar[r] &1.
}
\end{align}

As $\overline{\sf{SAut}(\A)}$ is a group, it also has a classifying groupoid $[\bullet/{\overline{\sf{SAut}(\A)}}]=(\overline{\sf{SAut}(\A)}\rrto \bullet)$. In the following, for a Lie groupoid $\K=(K^1\rrto K^0)$, we write a strict Lie groupoid morphism $\K\rto (\overline{\sf{SAut}(\A)}\rrto \bullet)$ as a morphism $K^1\rto \overline{\sf{SAut}(\A)}$ for simplicity.

\subsection{Symplectic orbifold groupoid}

We will also consider symplectic structure over orbifold groupoids (see for example \cite{Du-Chen-Wang2018}). Let $\A=(A^1\rrto A^0)$ be an orbifold groupoid. A {\bf symplectic form} over $\A$ is a symplectic form $\omega$ over $A^0$ such that $s^*\omega=t^*\omega$ over $A^1$. A {\bf symplectic orbifold groupoid} is an orbifold groupoid with a symplectic form.

Let $(\A,\omega)$ be a symplectic orbifold groupoid. An automorphism $\f\in\mathrm{SAut}^0(\A)$ is called a {\bf symplectomorphism}, if $f^{0,*}\omega=\omega$ over $A^0$. We also say $\f$ preserves the symplectic form.

\section{Locally topological product Lie groupoid fibrations}\label{S 3}

\subsection{Locally topological product fiber bundles}

\begin{defn}
A {\bf topological product fiber bundle} with fiber $\A=(A^1\rrto A^0)$ is a strict Lie groupoid morphism $\phi=(\phi^0,\phi^1):\G=(G^1\rrto G^0)\rto\K=(K^1\rrto K^0)$ such that
\begin{enumerate}
\item  $G^i=A^i\times K^i$ and $\phi^i$ is the projection to the second factor for $i=0,1$, and
\item  the unit map satisfies
$
u_\G=u_\A\times u_\K
$, and

\item  the source map satisfies
$
s_{\G}=s_\A\times s_\K
$, and
\item  the target map satisfies $t_\G(\alpha, 1_x)=(t_\A(\alpha),x)$ for every $\alpha\in A^1$ and $x\in K^0$.
\end{enumerate}
Then one can see that for every object $x\in K^0$, the fiber of $\phi$ over $x$
\[
\left((\phi^1)\inv(1_x)\rrto (\phi^0)\inv(x)\right)=(A^1\times\{1_x\}\rrto A^0\times\{x\})
\]
is isomorphic to $\A$ as a Lie groupoid via projecting $A^1\times\{1_x\}$ to $A^1$ and $A^0\times\{x\}$ to $A^0$.

A strict Lie groupoid morphism $\G\xrightarrow{\phi}\K$ is called a {\bf topologically trivial fiber bundle} with fiber $\A$ if $\G$ is isomorphic to a topological product fiber bundle $\G'\xrightarrow{\phi'} \K$ with fiber $\A$, i.e. we have the commutative diagram
\[
\xymatrix{\G\ar[r]^-\cong \ar[d]_-{\phi}& \G'\ar[d]^{\phi'} \\
\K\ar@{=}[r] &\K.}
\]
A strict Lie groupoid morphism $\G\xrightarrow{\phi}\K$ is called a {\bf locally topological product fiber bundle} with fiber $\A$ if
there exists an open cover $\mc U=\{U_a\}_{a\in\scr A}$ of $K^0$ such that the pullback strict morphism
$
    \sq_{\mc U}^*\phi:\G[\phi^*\mc U]\rto\K[\mc U]
$
    is a topologically trivial fiber bundle with fiber $\A$.
\end{defn}

From this definition we see that
the  {kernel}\footnote{The kernel $\ker\phi$ is a subgroupoid of $\G$. So we have the restriction morphism $\phi:\ker\phi\rto\K$ with images being $(u(K^0)\rrto K^0)=(K^0\rrto K^0)$, the trivial groupoid that representing the manifold $K^0$. So we may view $\ker \phi$ as a groupoid fiber bundle over $K^0$ and write it as $\phi:\ker\phi\rto K^0$.} of a locally topological product fiber bundle $\G\xrightarrow{\phi}\K$,
\begin{align}\label{E def-ker-phi}
\ker\phi:=(\ker\phi^1\rrto G^0)=\phi\inv(u(K^0)\rrto K^0),
\end{align}
is a locally trivial bundle of groupoids over $K^0$ with fiber isomorphic to $\A$.

\begin{defn}
Two topologically trivial fiber bundles with fiber $\A$, $\G_i\xrightarrow{\phi_i}\K_i,i=1,2,$ are {\bf isomorphic} if there is a groupoid isomorphism $\G_1\rto \G_2$ satisfying the following commutative diagram

\[
\xymatrix{
\ker\phi_1\ar@{^(->}[r] \ar@{=}[d] &
\G_1\ar[d]^-\cong \ar[r]^-{\phi_1} &
\K_1\ar@{=}[d]\\
\ker\phi_2\ar@{^(->}[r] &
\G_2\ar[r]^-{\phi_2}&
\K_2.}
\]

Let $\G_i \xrightarrow{\phi_i}\K_i$, $i=1,2$ be two locally topological product fiber bundles with fiber $\A$. We say that they are {\bf equivalent}, if there are refinements $\mc U_i,i=1,2,$ of
$\K_i$ such that the topologically trivial bundles $\G_1[\mc U_1]$ and $\G_2[\mc U_2]$ are isomorphic.
\end{defn}

It is direct to see that this is an equivalence relation between all locally topological product fiber bundles.

\subsection{Locally topological product Lie groupoid fibrations}

Given a topological product fiber bundle $\G=\A\times\K\xrightarrow{\phi}\K$, for each arrow $\xi:x\rto y$ in $K^1$, it induces
a strict morphism from $(\A,x)$ to $(\A,y)$ via
\begin{eqnarray}
&& \Lambda_{\xi}^0: (a,x)\xrightarrow{(1_a,\xi)} (b,y); \;\;\; b=t_\G(1_a,\xi)\\
&& \Lambda_\xi^1: (\alpha,1_x)\rto (\beta,1_y); \;\;\;
\beta=(1_{s_\A(\alpha)}, \xi)^{-1}(\alpha,1_x)
(1_{t_\A(\alpha)}, \xi).
\end{eqnarray}
This induces a morphism from $\A$ to $\A$
$$
\A=(\A,x)\xrightarrow{(\Lambda^0_\xi,\Lambda^1_\xi)}(\A,y) =\A,
$$
which we still denote  by $\Lambda_\xi$. Hence we have a
map
\begin{align}\label{E pre-action-map}
\Lambda: K^1\rto \mathrm{SMor}^0(\A,\A);\;\;\; \xi\mapsto \Lambda_\xi.
\end{align}
From the definition of topological product fiber bundle one can easily see that
\begin{lemma}\label{L pre-action-map}
(1)
The map $\Lambda$ is smooth in the sense that the induced map $K^1\times \A\rto \A$ is smooth,
(2)
$\Lambda_{1_x}=\sf id_\A$ for every $1_x\in K^1$,
(3)
$\Lambda$ induces a morphism
$
\bar \Lambda=\pi\circ\Lambda:K^1\rto \overline{\sf SMor(A,A)}.
$
\end{lemma}

\begin{defn}
We call this $\Lambda$ the {\bf pre-action map} of $\G$ if the image of $\Lambda$ is in $\mathrm{SAut}^0(\A)\subseteq \mathrm{SMor}^0(\A,\A)$. In this case we call $\G=\A\times\K\xrightarrow{\phi}\K$ a {\bf topological product Lie groupoid fibration} with fiber $\A$ (or by $\A$), and the morphism
$
\bar \Lambda:K^1\rto\overline{\sf SAut(A)}
$
the {\bf band} of $\G=\A\times \K\xrightarrow{\phi}\K$.

A topologically trivial fiber bundle $\G\rto\K$ is called a {\bf topologically trivial Lie groupoid fibration} with fiber $\A$ (or by $\A$) if it is isomorphic to a topological product Lie groupoid fibration $\G'\xrightarrow{\phi'}\K$ with fiber $\A$. We call the band of $\G'\xrightarrow{\phi'}\K$ the band of $\G\rto \K$.
\end{defn}

\begin{defn}\label{D locally topological product fibration}
A locally topological product fiber bundle $\phi:\G\rto \K$ with fiber $\A$ is called {\bf a locally topological product Lie groupoid fibration} with fiber $\A$ (or by $\A$) if there is a refinement $\G[\phi^*\mc U]\rto \K[\mc U]$ that is a topologically trivial Lie groupoid fibration with fiber $\A$. We write the fibration as
\begin{equation}
1\rto \A\rto \G\xrightarrow{\phi}\K\rto 1.
\end{equation}
We call the band of $\G[\phi^*\mc U]\rto\K[\mc U]$ the band of $\G\rto\K$.
\end{defn}

\begin{defn}\label{D equivalence-extension}
Two topologically trivial Lie groupoid fibrations with fiber $\A$ over $\K$, $\G_i\xrightarrow{\phi_i}\K, i=1,2$ are {\bf isomorphic} if the underlying topologically trivial fiber bundle are isomorphic.

Two locally topological product Lie groupoid fibrations with fiber $\A$ over $\K$, $\G_i\xrightarrow{\phi_i} \K,i=1,2$ are {\bf equivalent} if the underlying locally topological product fiber bundles are equivalent.
\end{defn}

\begin{example}\label{Example 1}
Let $A$ and $K$ be two finite groups and $G$ be a group extension of $K$ by $A$, i.e. $1\rto A\xrightarrow{i}G\xrightarrow{j}K\rto 1$. Suppose $G\cong K\ltimes_{(\omega,f)}A$ with $\omega:K\rto \mathrm{Aut}(A)$ and $f: K\times K\rto A$, where $\omega$ is a lifting of the band $K\rto\mathrm{Out}(A)$ of the group extension and $f$ is the associated cofactor. Suppose there is a smooth $K$-action\footnote{Here we use right action for convenience with our notation of multiplication of arrows in Lie groupoids.} on a smooth manifold $X$. Let $G$ act on $X$ via $G\xrightarrow{j} K$. So $A$ acts trivially on $X$ via $A\xrightarrow{i} G$. Let $G\ltimes X$ and $A\ltimes X$ be the corresponding action/translation groupoids. They are both Lie groupoids. Consider the Lie groupoid morphism $\phi=(\phi^0,\phi^1): G\ltimes X\rto [\bullet/K]$, where $\phi^0$ is the projection $X\rto \bullet$ and $\phi^1$ is $G\times X\xrightarrow{proj_1} G\xrightarrow{j} K$. It is obvious that $G\ltimes X$ is isomorphic to $(K\times_{(\omega,f)}A)\ltimes X$, and $G\ltimes X\rto [\bullet/K]$ is a topologically trivial Lie groupoid fibration with fiber $A\ltimes X$ since it is isomorphic to
\begin{align}\label{E K-times-A-ltimes X}
(K\times_{(\omega,f)}A)\ltimes X\rrto [\bullet/K],
\end{align}
a topological product Lie groupoid fibration with fiber $A\ltimes X$.
\end{example}
This example studies (locally) topological product Lie groupoid fibration over the classifying groupoids, which can be seen as a counterpart of gerbes, which are (locally) topological product Lie groupoid fibrations with fiber being classifying groupoids. Next we consider another example that constructs a new (locally) topological product Lie groupoid fibration out of certain groupoid action on a smooth manifold.

\begin{example}
Suppose we have a topological product Lie groupoid fibration $\phi:\G=\A\times \K\rto\K$. Let $X$ be a smooth manifold. Suppose $\G$ acts on $X$, whose anchor map is
\[
\pi:X\rto G^0=A^0\times K^0
\]
such that $X$ is a trivial fiber bundle with fiber being a smooth manifold $F$, and action map is
\[
\mu:G^1\times_{s,\pi}X\rto X.
\]
Then we get an action groupoid $\G\ltimes X=(G^1\times_{s,\pi}X\rrto X)$.

By assumption, $X=F\times A^0\times K^0$ with $\pi:X\rto G^0=A^0\times K^0$ being the projection to second factor, and
\[
G^1\times_{s,\pi}X=(A^1\times K^1)\times_{s,\pi}F\times A^0\times K^0=A^1\times K^1\times F=G^1\times F.
\]
So we have a projection $\psi=(\psi^0,\psi^1):\G\ltimes X\rto \K$ by
\[
\psi^0:X=F\times A^0\times K^0\rto K^0 \qq \mathrm{and} \qq
\psi^1:G^1\times F\rto K^1.
\]
Then one can see that $\phi:\G\ltimes X\rto \K$ is a topological product Lie groupoid fibration with fiber $\A\ltimes F$.

Generally, if $\phi:\G\rto \K$ is a locally topological product Lie groupoid fibration with fiber $\A$ and acts on a smooth manifold $X$, and there is an open cover $\mc U$ of $K^0$ such that $\G[\phi^*\mc U]\rto \K[\mc U]$ is a topological product Lie groupoid fibration and $X\rto G^0$ is trivialized over $\phi^*\mc U$. Then $\G\ltimes X\rto \G\rto \K$ is a locally topological product Lie groupoid fibration with fiber $\A\ltimes F$, where $F$ is the fiber of $X\rto G^0$. For example when $\K$ is a proper \'etale Lie groupoid, $\G$ is an \'etale gerbe over $\K$ and $X\rto G^0=K^0$ is locally trivial fibration, the assumptions on $\G$ and $X$ are satisfied.
\end{example}

\subsection{Generalized cocycles}

In this subsection we assume that $\G=\A\times \K\xrightarrow{\phi}\K$ is a topological product Lie groupoid fibration with pre-action map $\Lambda: K^1\rto \mathrm{SAut}^0(\A)$ given by \eqref{E pre-action-map}. For simplicity, for every arrow $\xi\in K^1$, we denote the corresponding strict automorphism by $\Lambda_\xi=(\Lambda_\xi,\Lambda_\xi)$. Recall that $\Lambda_{1_x}=\sf id_A$ for all $x\in K^0$.

In general the pre-action map $\Lambda:K^1\rto \mathrm{SAut}^0(\A)$ is not a morphism. The default of $\Lambda$ being a homomorphism is measured by the following smooth map:
\begin{align}\label{E def-Omega}
\begin{split}
\Omega:K^{[2]}\times A^0=&(K^1\times_{t,K^0,s}K^1)\times A^0\rto G^1,\,\,
\Omega(\xi,\eta,a):=
(\xi,1_a)\cdot(\eta,1_{\Lambda_{\xi}(a)})
\cdot(\xi\eta,1_{\Lambda_{\xi\eta}\inv\Lambda_\eta\Lambda_\xi(a)})\inv.
\end{split}\end{align}
We call $\Omega$ the {\bf cofactor} associated to the pre-action map $\Lambda$.

The images of $\Omega$ lie in the kernel of $\G=\A\times \K\xrightarrow{\phi} \K$, since
\[
\phi^1(\Omega(\xi,\eta,a))=\xi\cdot \eta \cdot(\xi\eta)\inv=1_{s(\xi)}.
\]
So in the rest of this subsection, we identify $\Omega$ as the map
\[
\Omega:K^{[2]}\times A^0\rto \ker\phi^1\rto A^1,
\]
and $\Omega(\xi,\eta,a)$ will denote an element in $A^1$ and also an element $(\Omega(\xi,\eta,a),1_{s(\xi)})$ in $G^1$.

We also have
\begin{align}\label{E Omega(alpha,1,a)=1}
\begin{split}
\Omega(\xi,1_{t(\xi)},a)
=(1_{s(\xi)},1_a)=1_{(s(\xi),a)},\qq
\Omega(1_{s(\xi)},\xi,a)&=(1_{s(\xi)},1_a)=1_{(t(\xi),a)}.
\end{split}
\end{align}

The associativity of multiplication in $G^1$ gives rise to some constraints on $\Lambda$ and $\Omega$, which we call the {\bf generalized cocycle condition}. We next write down this condition explicitly.

First of all take three composable arrows $\xi,\eta,\zeta\in K^1$, and suppose $x\xrightarrow \xi   y\xrightarrow \eta z\xrightarrow \zeta w$. Take an object $a \in A^0$. We have the following diagram of arrows in $G^1=K^1\times A^1$
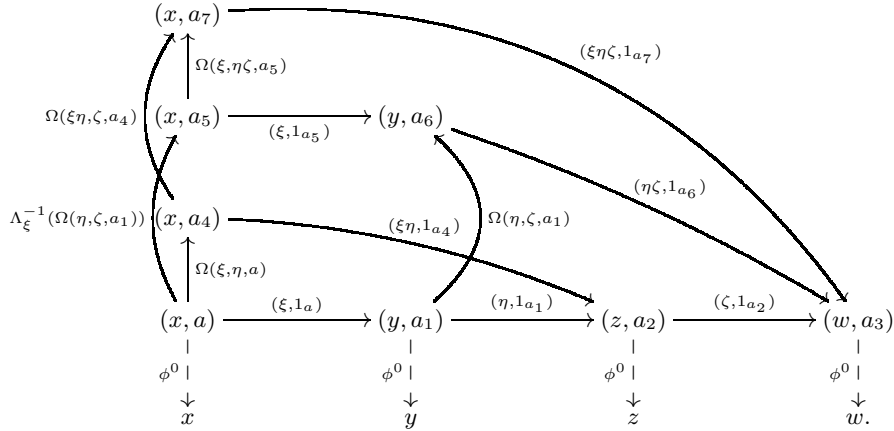
\begin{align*}
\small{\xymatrix{
(x,a_7) \ar@/^3.5pc/[dddrrrrrr]^-{\qq\,(\xi\eta\zeta,1_{a_7})} &&&&&&\\
(x,a_5)\ar[rr]_-{(\xi,1_{a_5})} \ar[u]_-{\Omega(\xi,\eta\zeta,a_5)}
&&
(y,a_6)\ar@/^0.5pc/[ddrrrr]^-{\,\,(\eta\zeta,1_{a_6})}&&&&\\
(x,a_4)\ar@/^0.8pc/[drrrr]^-{(\xi\eta,1_{a_4})}
\ar@/^1.35pc/[uu]^-{\Omega(\xi\eta,\zeta,a_4)}&&
&&&\\
(x,a)\ar[rr]^-{(\xi,1_a)} \ar[u]_-{\Omega(\xi,\eta,a)}
\ar@/^1.1pc/[uu]^-{\Lambda_{\xi}\inv (\Omega(\eta,\zeta,a_1))}
\ar@{-->}[d]_-{\phi^0}
&&(y,a_1)\ar[rr]^{(\eta,1_{a_1})}
\ar@/_2.2pc/[uu]_-{\Omega(\eta,\zeta,a_1)}
\ar@{-->}[d]_-{\phi^0}
&& (z,a_2)\ar[rr]^-{(\zeta,1_{a_2})}
\ar@{-->}[d]_-{\phi^0}
&& (w,a_3) \ar@{-->}[d]_-{\phi^0} \\
x && y && z && w.}}
\end{align*}
Therefore
\[
\Lambda_{\xi\eta}\inv \Lambda_{\eta}\Lambda_{\xi}(a)=a_4,\qq
\Lambda_{\eta\zeta}\inv\Lambda_{\zeta}\Lambda_{\eta}(a_1)=a_6,\qq
\Lambda_{\xi\eta\zeta}\inv\Lambda_{\zeta}\Lambda_{\eta}\Lambda_{\xi}(a) =a_7.
\]
Then
\begin{eqnarray*}
&&\Omega(\xi,\eta,a)\cdot\Omega(\xi\eta,\zeta,a_4)\cdot (\xi\eta\zeta,1_{a_7}) \\
&=&(\xi,1_a)\cdot(\eta,1_{a_1})\cdot(\xi\eta,1_{a_4})\inv
\cdot (\xi\eta,1_{a_4})
\cdot(\zeta,1_{a_2})\cdot(\xi\eta\zeta,1_{a_7})\inv\cdot (\xi\eta\zeta,1_{a_7}) \\
&=&(\xi,1_a)\cdot(\eta,1_{a_1})\cdot(\zeta,1_{a_2})\\
&=&(\xi,1_a)\cdot\Omega(\eta,\zeta,a_1)\cdot(\eta\zeta,1_{a_6})\\
&=&(\xi,1_a)\cdot\Omega(\eta,\zeta,a_1)\cdot(\xi,1_{a_5})\inv \cdot (\xi,1_{a_5})\cdot(\eta\zeta,1_{a_6})\\
&=&(\xi,1_a)\cdot\Omega(\eta,\zeta,a_1) \cdot(\xi,1_{a_5})\inv
\cdot \Omega(\xi,\eta\zeta,a_5)\cdot (\xi\eta\zeta,1_{a_7}),
\end{eqnarray*}
which is
\begin{align}\label{E generalized-cocycle-1}
\Omega(\xi,\eta,a) \cdot \Omega(\xi\eta,\zeta,a_4)=&
\Lambda_\xi\inv(\Omega(\eta,\zeta,a_1))\cdot\Omega(\xi,\eta\zeta,a_5).
\end{align}
As a special case, let $\eta=\xi\inv$ with $a_1=\Lambda_\xi(a)$ and $a_2=\Lambda_{\xi\inv}(a_1)$; in addition with \eqref{E Omega(alpha,1,a)=1} we get
\begin{align}\label{E generalized-cocycle-2}
\Omega(\xi,\xi\inv,a)=
\Lambda_\xi\inv(\Omega(\xi\inv,\xi,a_1)). \end{align}
We call \eqref{E generalized-cocycle-1} the {\bf generalized cocycle condition}, and $(\Lambda,\Omega)$ a {\bf generalized cocycle}.

The analysis above shows that for a topological product Lie groupoid fibration $\G=\A\times\K\xrightarrow{\phi}\K$, the pre-action map $\Lambda:K^1\rto \mathrm{SAut}^0(\A)$ and the band $\bar \Lambda:K^1\rto \overline{\sf{SAut}(\A)}$ fit into the following diagram
\begin{align}\label{D band-lifting}
\begin{split}
\xymatrix{
& && K^1\ar[d]^-{\Lambda}\ar[dr]^-{\bar \Lambda} &&\\
1\ar[r] &Z_\A\ar@{^(->}[r] &N_\A\ar[r]^-{t_{saut}} & \mathrm{SAut}^0(\A)\ar[r]^-{\pi} &\overline{\sf{SAut}(\A)} \ar[r]&1.}
\end{split}
\end{align}
The obstruction of $\Lambda$ being a morphism is recorded by the smooth map $\Omega:K^{[2]}\times A^0\rto G^1$ (cf. \eqref{E def-Omega}). The pair $(\Lambda,\Omega)$ satisfies the generalized cocycle condition \eqref{E generalized-cocycle-1}.

\begin{example}\label{Example 2}
Consider the topological product Lie groupoid fibration in Example \ref{Example 1}, i.e. $(K\ltimes_{(\omega,f)}A)\ltimes X\rrto[\bullet/K]$ in \eqref{E K-times-A-ltimes X}. The pre-action map $\Lambda$ associated to this topological product Lie groupoid fibration $(K\ltimes_{(\omega,f)}A)\ltimes X\rrto[\bullet/K]$ is $\Lambda_k:A\ltimes X\rto A\ltimes X$ with
\[
\Lambda_k^0:X\rto X, \qq x\mapsto x\cdot k.
\]
being the $k$-action on $X$, and
\[
\Lambda^1_k:A\times X\rto A\times X, \qq (a,x)\mapsto (\omega(k)(a),x\cdot k).
\]
So $\Lambda^1_k=(\omega(k),\Lambda_k^0)$. Therefore the corresponding cofactor $\Omega:K\times K\times X\rto (K\ltimes_{(\omega,f)}A)\times X$ is determined by $f:K\times K\rto A$, that is
\[
\Omega(k_1,k_2,x)=(1,f(k_1,k_2),x).
\]
\end{example}

\section{Existence and obstruction of topological product Lie groupoid fibrations}\label{S 4}

In this section we study the existence and obstruction of topological product Lie groupoid fibrations. Let $\K$ and $\A$ be two Lie groupoids. Suppose we have a morphism  $\bar \Lambda:K^1\rto \overline{\sf SAut(A)}$ and a smooth lifting $\Lambda$ of $\bar \Lambda$ (cf. \eqref{D band-lifting}) satisfying $\Lambda_{1_x}:=\Lambda(1_x)=\sf id_\A$ for all $x\in K^0$. Recall from item (1) in Lemma \ref{L pre-action-map} that the lifting $\Lambda$ is smooth means that the induced map $K^1\times \A\rto \A$ is smooth.

\subsection{Existence of topological product Lie groupoid fibrations}
In this subsection we study the question when there is a topological product Lie groupoid fibration with fiber $\A$ over $\K$ whose band is $\bar\Lambda$.

Since $\Lambda$ is a lifting of $\bar\Lambda$, for two composable arrows $\xi,\eta\in K^1$ we have
\[
\pi(\Lambda_{\xi\eta}\inv\circ\Lambda_{\eta}\circ\Lambda_{\xi})=[\sf id_A]\in \overline{\sf{SAut}(\A)}.
\]
Therefore there is a natural transformation $\Phi_{\xi,\eta}:{\sf id}_\A\Rto \Lambda_{\xi\eta}\inv\circ\Lambda_{\eta}\circ\Lambda_{\xi}:\A\rto\A$. We define
\begin{align}\label{E def-Omega-2}
\begin{split}
&\Omega:K^{[2]}\times A^0\rto  A^1,
\qq 
\Omega(\xi,\eta,a):=\Phi_{\xi,\eta}(a).
\end{split}
\end{align}
In particular, when $\xi=1_x$ or $\eta=1_y$ we take $\Phi_{1_x,\eta}=\Phi_{\xi,1_y}=1_{\sf id_A}$, and then we have
\begin{align}\label{E Omega(1,alpha)=Omega(alpha,1)=1}
\Omega(1_x,\eta,a)=1_a,\qq \Omega(\xi,1_y,a)=1_a.
\end{align}

For each $\xi\in K^1$, we omit the superscripts and write the corresponding automorphism $\Lambda_\xi$ of $\A$ by
\[
\Lambda_\xi:=(\Lambda_\xi,\Lambda_\xi)
\]
for simplicity. Therefore we have the commutative diagram
\[
\xymatrix{
a \ar[rr]^-{\Omega(\xi,\eta,a)} \ar[d]_-{\alpha} &&
\Lambda_{\xi\eta}\inv\Lambda_\eta\Lambda_\xi(a)
\ar[d]^-{\Lambda_{\xi\eta}\inv\Lambda_\eta\Lambda_\xi(\alpha)}\\
b\ar[rr]^-{\Omega(\xi,\eta,b)} & &
\Lambda_{\xi\eta}\inv\Lambda_\eta\Lambda_\xi(b),}
\]
for every arrow $\alpha\in A^1$. So for every arrow $\alpha\in A^1$ we have
\begin{align}\label{E 3.3}
\Lambda_{\xi\eta}\inv\Lambda_\eta\Lambda_\xi(\alpha)=
\Omega(\xi,\eta,s(\alpha))\inv\cdot \alpha \cdot \Omega(\xi,\eta,t(\alpha)).
\end{align}

\begin{theorem}\label{T S-Omega-def-G-SOmega}
Let $\Lambda:K^1\rto \mathrm{SAut}^0(\A)$ be a smooth lifting of the morphism $\bar \Lambda:K^1\rto \overline{\sf{SAut}(\A)}$. Suppose that $\Omega$, defined in \eqref{E def-Omega-2}, is smooth. If $(\Lambda,\Omega)$ satisfies the generalized cocycle condition \eqref{E generalized-cocycle-1}, i.e.
\begin{align}\label{E generalized-cocycle-theorem-4}
\Omega(\xi,\eta,a) \cdot \Omega(\xi\eta,\zeta,a_4)
=\Lambda_\xi\inv(\Omega(\eta,\zeta,a_1))\cdot\Omega(\xi,\eta\zeta,a_5)
\end{align}
for every composable arrows $\xi,\eta$ in $K^1$ and objects $a\in A^0$, then there is a topological product Lie groupoid fibration with fiber $\A$ over $\K$, $\G\rto \K$, whose band is $\bar \Lambda$. Denote this topological product Lie groupoid fibration by $\A\rtimes_{\Lambda,\Omega}\K$.
\end{theorem}
\begin{proof}
We set $\G=\A\times \K=(A^1\times K^1\rrto A^0\times K^0)$. We next define the structure maps:
\begin{enumerate}
\item The source and target maps for $\G$ are
\begin{align}
s_\G(\alpha,\xi)&:=(s(\alpha),s(\xi)),\label{E def-s-in-K-ltimes-A}\\
t_\G(\alpha,\xi)&:=(t(\Lambda_\xi(\alpha)),t(\xi))= (\Lambda_\xi(t(\alpha)),t(\xi)),\label{E def-t-in-K-ltimes-A}
\end{align}
where the $s$ and $t$ on the right hand side are the source and target maps for $\K$ and $\A$. We do not use subscripts to distinguish them since obviously there is no ambiguity.

\item The multiplication over $A^1\times K^1$ is
\begin{align}\label{E def-multi-T-4-1}
\begin{split}
(\alpha,\xi)\cdot(\beta,\eta)
:=\left( \Omega(\xi,\eta,s(\alpha))\cdot\Lambda_{\xi\eta}\inv \Lambda_\eta\Lambda_\xi(\alpha)\cdot \Lambda_{\xi\eta}\inv \Lambda_\eta(\beta),\xi\cdot\eta\right),
\end{split}
\end{align}
where the multiplications $ \Omega(\xi,\eta,s(\alpha))\cdot\Lambda_{\xi\eta}\inv \Lambda_\eta\Lambda_\xi(\alpha)\cdot \Lambda_{\xi\eta}\inv \Lambda_\eta(\beta)$ and $\xi\cdot\eta$ are taken in $\A$ and $\K$ respectively, and when there is no ambiguity we will omit the ``$\cdot$'' in the following computation.

\item The unit map is $u(x,a):=(1_x,1_a)$.

\item The inverse arrow of an arrow $(\alpha,\xi)$ is
\[
(\alpha,\xi)\inv:=\left(\Lambda_\xi(\alpha\inv)\cdot \Lambda_{\xi\inv}\inv [\Omega(\xi,\xi\inv,s(\alpha))\inv],\xi\inv\right).
\]
\end{enumerate}

Note that $\Lambda$ is a smooth lifting of $\bar\Lambda$. As we remarked in (1), Lemma \ref{L pre-action-map}, the smoothness of $\Lambda$ means that the induced map $K^1\times \A\rto \A$ is smooth. Then by the assumption that $\Omega$ is smooth, one can see that all these structure maps above are smooth. On the other hand, since $s_\K, s_\A, t_\K, t_\A$ are all submersions and $\Lambda_\xi$ is an automorphism of $\A$, we see that both $s_\G$ and $t_\G$ are submersions.

We next show that these structure maps actually give rise to a Lie groupoid structure over $\G$. We only have to show that these structure maps satisfy the axioms for $\G$ being a category. Once we prove this, it is obvious that the projections to the first factors give rise to a topological product Lie groupoid fibration $\G\rto \K$ with fiber $\A$ over $\K$.

We first verify the associativity of the multiplication \eqref{E def-multi-T-4-1}. Take three composable arrows $(\alpha,\xi)$, $(\beta,\eta)$, $(\gamma,\zeta)\in G^1$. We have
\begin{eqnarray*}
&&
\left((\alpha,\xi)\cdot(\beta,\eta)\right)\cdot(\gamma,\zeta)\\
&\stackrel{\eqref{E def-multi-T-4-1}}{=}&\left( \Omega(\xi,\eta,s(\alpha))\cdot\Lambda_{\xi\eta}\inv\Lambda_\eta \Lambda_\xi(\alpha)\cdot\Lambda_{\xi\eta}\inv\Lambda_\eta(\beta), \xi\eta \right)\cdot(\gamma,\zeta)\\
&\stackrel{\eqref{E def-multi-T-4-1}}{=}&
\Big(\Omega(\xi\eta,\zeta,s(\alpha))\cdot \Lambda_{\xi\eta\zeta}\inv \Lambda_{\zeta}\Lambda_{\xi\eta} \Big[\Omega(\xi,\eta,s(\alpha))
\cdot\Lambda_{\xi\eta}\inv \Lambda_\eta \Lambda_\xi(\alpha) \cdot  \Lambda_{\xi\eta}\inv\Lambda_\eta(\beta)\Big] \cdot  \Lambda_{\xi\eta\zeta}\inv \Lambda_\zeta(\gamma), \xi\eta\zeta\Big)\\
&=&\Big( \Omega(\xi\eta,\zeta,s(\alpha))\cdot \Lambda_{\xi\eta\zeta}\inv \Lambda_\zeta\Lambda_{\xi\eta} [\Omega(\xi,\eta,s(\alpha))]\cdot
\Lambda_{\xi\eta\zeta}\inv \Lambda_\zeta\Lambda_\eta\Lambda_\xi(\alpha)\cdot \Lambda_{\xi\eta\zeta}\inv\Lambda_\zeta\Lambda_\eta(\beta)\cdot \Lambda_{\xi\eta\zeta}\inv\Lambda_\zeta(\gamma),\xi\eta\zeta
\Big)
\end{eqnarray*}
and
\begin{eqnarray*}
&&
(\alpha,\xi)\cdot \left((\beta,\eta) \cdot (\gamma,\zeta)\right)\\
&\stackrel{\eqref{E def-multi-T-4-1}}{=}& (\alpha,\xi)\cdot  \left(\Omega(\eta,\zeta,s(\beta))\cdot\Lambda_{\eta\zeta}\inv \Lambda_\zeta\Lambda_\eta(\beta)\cdot\Lambda_{\eta\zeta}\inv \Lambda_\zeta(\zeta),\eta\zeta\right)\\
&\stackrel{\eqref{E def-multi-T-4-1}}{=}&
\Big( \Omega(\xi,\eta\zeta,s(\alpha))\cdot \Lambda_{\xi\eta\zeta}\inv \Lambda_{\eta\zeta}\Lambda_\xi(\alpha)
\cdot \Lambda_{\xi\eta\zeta}\inv \Lambda_{\eta\zeta}\left[ \Omega(\eta,\zeta,s(\beta))\cdot\Lambda_{\eta\zeta}\inv \Lambda_\zeta\Lambda_\eta(\beta)\cdot\Lambda_{\eta\zeta}\inv \Lambda_\zeta(\zeta)\right], \xi\eta\zeta\Big)\\
&=&\Big( \Omega(\xi,\eta\zeta,s(\alpha))\cdot  \Lambda_{\xi\eta\zeta}\inv \Lambda_{\eta\zeta}\Lambda_\xi(\alpha)
\cdot  \Lambda_{\xi\eta\zeta}\inv \Lambda_{\eta\zeta}[\Omega(\eta,\zeta,s(\beta))] \cdot \Lambda_{\xi\eta\zeta}\inv \Lambda_\zeta\Lambda_\eta(\beta)\cdot \Lambda_{\xi\eta\zeta}\inv \Lambda_\zeta(\gamma), \xi\eta\zeta\Big).
\end{eqnarray*}

It is direct to see that the associativity follows from the following equality
\begin{align}\label{E pre-associativity}
\begin{split}
&
\Omega(\xi\eta,\zeta,s(\alpha))\cdot \Lambda_{\xi\eta\zeta}\inv \Lambda_\zeta\Lambda_{\xi\eta} [\Omega(\xi,\eta,s(\alpha))]\cdot
\Lambda_{\xi\eta\zeta}\inv \Lambda_\zeta\Lambda_\eta\Lambda_\xi(\alpha)\\
=\,&\Omega(\xi,\eta\zeta,s(\alpha))\cdot  \Lambda_{\xi\eta\zeta}\inv \Lambda_{\eta\zeta}\Lambda_\xi(\alpha)
\cdot  \Lambda_{\xi\eta\zeta}\inv \Lambda_{\eta\zeta}[\Omega(\eta,\zeta,s(\beta))].
\end{split}
\end{align}
To prove \eqref{E pre-associativity}, we apply \eqref{E 3.3} to both sides. We have
\begin{eqnarray*}
\mathrm{LHS}
&=&\Omega(\xi\eta,\zeta,s(\alpha))
\cdot   \Lambda_{\xi\eta\zeta}\inv \Lambda_\zeta\Lambda_{\xi\eta} \left[\Omega(\xi,\eta,s(\alpha))\cdot \Lambda_{\xi\eta}\inv \Lambda_\eta \Lambda_\xi(\alpha)\right]
\\&
\stackrel{\eqref{E 3.3}}{=}& \Omega(\xi\eta,\zeta,s(\alpha))\cdot \Big\{\Omega(\xi\eta,\zeta,s(\alpha))\inv
\cdot \left[\Omega(\xi,\eta,s(\alpha))
\cdot \Lambda_{\xi\eta}\inv \Lambda_\eta \Lambda_\xi(\alpha)\right]
\cdot \Omega(\xi\eta,\zeta,t(\Lambda_{\xi\eta}\inv\Lambda_\eta  \Lambda_\xi(\alpha)))\Big\}\\
&=&\left[\Omega(\xi,\eta,s(\alpha))
\cdot \Lambda_{\xi\eta}\inv \Lambda_\eta \Lambda_\xi(\alpha)\right]
\cdot \Omega(\xi\eta,\zeta,t(\Lambda_{\xi\eta}\inv\Lambda_\eta  \Lambda_\xi(\alpha)))\\
&\stackrel{\eqref{E 3.3}}{=} & \alpha\cdot \Omega(\xi,\eta,t(\alpha))\cdot \Omega(\xi\eta,\zeta,t(\Lambda_{\xi\eta}\inv\Lambda_\eta  \Lambda_\xi(\alpha)))
\end{eqnarray*}
and
\begin{eqnarray*}
\mathrm{RHS}&=&\Omega(\xi,\eta\zeta,s(\alpha))
\cdot\Lambda_{\xi\eta\zeta}\inv\Lambda_{\eta\zeta}\Lambda_\xi \left[\alpha\cdot\Lambda_\xi\inv(\Omega(\eta,\zeta,s(\beta)))\right]\\
&\stackrel{\eqref{E 3.3}}{=} & \Omega(\xi,\eta\zeta,s(\alpha))\cdot \Big\{\Omega(\xi,\eta\zeta,s(\alpha))\inv\cdot\alpha
\cdot\Lambda_\xi\inv[\Omega(\eta,\zeta,s(\beta))]\cdot \Omega(\xi,\eta\zeta,t(\Lambda_\xi\inv[ \Omega(\eta,\zeta,s(\beta))]))\Big\}\\
&=&\alpha\cdot \Lambda_\xi\inv[\Omega(\eta,\zeta,s(\beta))]\cdot \Omega(\xi,\eta\zeta,t(\Lambda_\xi\inv[ \Omega(\eta,\zeta,s(\beta))])).
\end{eqnarray*}
Then since $(\Lambda,\Omega)$ satisfies \eqref{E generalized-cocycle-theorem-4} we get
$
\mathrm{LHS}=\mathrm{RHS}.
$
Hence the multiplication is associative.

Now we show that the unit for this multiplication over $(a,x)$ is $u(a,x)=(1_a,1_x)$. We have
\begin{eqnarray*}
(1_a,1_x)\cdot (\alpha,\xi)
&=&\left(\Omega(1_x,\xi,a)\cdot\Lambda_{1_x\xi}\inv \Lambda_\xi\Lambda_{1_x}(1_a)\cdot\Lambda_{1_x\xi}\inv \Lambda_\xi(\alpha), 1_x\xi\right)\\
&=&(\Omega(1_x,\xi,a)\cdot\alpha,\xi)=(\alpha,\xi)
\end{eqnarray*}
where we have used $\Lambda_{1_x}=\sf id_\A$ and $\Omega(1_x,\xi,a)=1_a$. Similarly, since $\Omega(\xi,1_y,a)=1_a$ we get
\begin{eqnarray*}
(\alpha,\xi)\cdot (1_b,1_y)
&=&\left(\Omega(\xi,1_y,a)\cdot\Lambda_{\xi 1_y}\inv \Lambda_{1_y}\Lambda_\xi(\alpha)\cdot\Lambda_{\xi 1_y}\inv \Lambda_{1_y}(1_b),\xi 1_y\right)\\
&=&(\Omega(\xi,1_y,a)\cdot  \alpha,\xi)=(\alpha,\xi).
\end{eqnarray*}

Finally we show that the inverse map and multiplication are compatible. Note that for any $\xi\in K^1$ we have $\Lambda_{\xi\inv\xi}=\Lambda_{1_{t(\xi)}}=\sf id_\A$ and $\Lambda_{\xi\xi\inv}=\Lambda_{1_{s(\xi)}}=\sf id_\A$. Therefore
\begin{eqnarray*}
(\alpha,\xi)\cdot (\alpha,\xi)\inv
&=&\Big(\Omega(\xi,\xi\inv,s(\alpha)) \cdot  {\sf id_A}\Lambda_{\xi\inv}\Lambda_\xi(\alpha)
\cdot  {\sf id_A}\Lambda_{\xi\inv}\left[\Lambda_\xi(\alpha\inv)\cdot \Lambda_{\xi\inv}\inv[ \Omega(\xi,\xi\inv,s(\alpha))\inv] \right], \xi\xi\inv\Big)\\
&=&\left(\Omega(\xi,\xi\inv,s(\alpha))\cdot \Lambda_{\xi\inv}\Lambda_\xi(\alpha\alpha\inv)\cdot \Omega(\xi,\xi\inv,s(\alpha))\inv, 1_{s(\xi)}\right)\\
&=&\left(\Omega(\xi,\xi\inv,s(\alpha)) \cdot  1_{\Lambda_{\xi\inv}\Lambda_\xi(s(\alpha))}\cdot \Omega(\xi,\xi\inv,s(\alpha))\inv, 1_{s(\xi)}\right)\\
&=&(1_{s(\alpha)},1_{s(\xi)}),
\end{eqnarray*}
and
\begin{eqnarray*}
(\alpha,\xi)\inv\cdot (\alpha,\xi)
&=&\Big( \Omega(\xi\inv,\xi,t(\Lambda_\xi(\alpha)))
\cdot   {\sf id_A} \Lambda_\xi\Lambda_{\xi\inv}\left[\Lambda_\xi(\alpha\inv)\cdot   \Lambda_{\xi\inv}\inv[\Omega(\xi,\xi\inv,s(\alpha))]\inv\right]\cdot  {\sf id_A}\Lambda_\xi(\alpha),\xi\inv\xi\Big)\\
&=&\Big( \Omega(\xi\inv,\xi,t(\Lambda_\xi(\alpha_)))
\cdot\Lambda_\xi\left[\Lambda_{\xi\inv}\Lambda_\xi(\alpha)\inv \cdot   \Omega(\xi,\xi\inv,s(\alpha))\inv\cdot  \alpha\right], \xi\inv\xi\Big)\\
&\stackrel{\eqref{E 3.3}}{=} &\Big( \Omega(\xi\inv,\xi,t(\Lambda_\xi(\alpha)))\cdot   \Lambda_\xi(\Omega(\xi,\xi\inv,t(\alpha))) , \xi\inv\xi\Big)\\
&\stackrel{\eqref{E generalized-cocycle-2}}{=} &(1_{t(\Lambda_\xi(\alpha))},1_{t(\xi)}).
\end{eqnarray*}
Therefore with structure maps defined above $\G$ is a Lie groupoid, and the projection to the second factor give rise to a topological product Lie groupoid fibration $\G\rto \K$.

It is obvious from the definition of structure maps that the pre-action map is $\Lambda$, hence the band is $\bar \Lambda$. This finishes the proof.
\end{proof}

We have
\begin{coro}\label{C G-cong-KSOA}
Let $\G\rto \K$ be a topological product Lie groupoid fibration over $\K$ with fiber $\A$. Suppose its pre-action map $\Lambda$ is given by \eqref{E pre-action-map} and the associated cofactor $\Omega$ is given by \eqref{E def-Omega}. Then there is an isomorphism of topological product Lie groupoid fibrations
$
\G\cong\A\rtimes_{\Lambda,\Omega}\K.
$
\end{coro}
\begin{proof}
The object part of this groupoid isomorphism is identity. Take an arrow $(\alpha,\xi)\in G^1=A^1\times K^1$. Suppose $t(\alpha,\xi)=b$. Then the arrow part of this groupoid isomorphism is
\[
(\alpha,\xi)\mapsto \left(\text{proj}_1\left[(\alpha,\xi)\cdot_\G (1_{\Lambda_\xi\inv(b)},\xi)\inv\right],\xi\right).
\]
The inverse on arrows is given by
$(\alpha,\xi)\mapsto (\alpha,1_{s(\xi)})\cdot_\G (1_{t(\alpha)},\xi).$ It is direct to verify that this is an isomorphism of topological product Lie groupoid fibrations over $\K$ with fiber $\A$.
\end{proof}

\subsection{Obstruction of topological product Lie groupoid fibrations}
Let $\Lambda$ be a smooth lifting of a morphism $\bar \Lambda:K^1\rto \overline{\sf{SAut}(\A)}$ and $\Omega:K^{[2]}\times A^0\rto A^1$ be a smooth family of natural transformations defined as \eqref{E def-Omega-2}. In this subsection we study the question when will $(\Lambda,\Omega)$ satisfy the generalized cocycle condition \eqref{E generalized-cocycle-1}.

Take an $a\in A^0$ and three composable arrows $\xi,\eta,\zeta$ in $K^1$, i.e. $(\xi,\eta,\zeta)\in K^{[3]}=K^1\times_{t,s}K^1\times_{t,s}K^1$.
Consider the following diagram of automorphisms and natural transformations
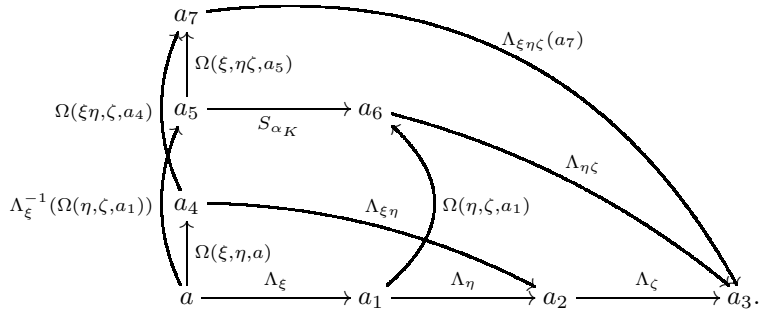
\[\xymatrix{
a_7\ar@/^3.5pc/[dddrrrrrr]^-{\Lambda_{\xi\eta\zeta}(a_7)} &&&&&&\\
a_5\ar[rr]_-{S_{\alpha_K}} \ar[u]_-{\Omega(\xi,\eta\zeta,a_5)}
&&
a_6\ar@/^0.8pc/[ddrrrr]^-{\Lambda_{\eta\zeta}}&&&&\\
a_4\ar@/^0.8pc/[drrrr]^-{\Lambda_{\xi\eta}}
\ar@/^0.8pc/[uu]^-{\Omega(\xi\eta,\zeta,a_4)}&&
&&&\\
a\ar[rr]^-{\Lambda_\xi} \ar[u]_-{\Omega(\xi,\eta,a)}
\ar@/^0.8pc/[uu]^-{\Lambda_\xi\inv(\Omega(\eta,\zeta,a_1))}
&&a_1\ar[rr]^{\Lambda_\eta}
\ar@/_2pc/[uu]_-{\Omega(\eta,\zeta,a_1)}
&& a_2\ar[rr]^-{\Lambda_\zeta} && a_3.}
\]
Take an arrow $\alpha\in\Gamma_a$, the isotropy group of $a$ in $\A$. Then by the definition of $\Omega$ we have the following equalities in $A^1$
\begin{eqnarray*}
&&\Omega(\xi\eta,\zeta,a_4)\inv \cdot \Omega(\xi,\eta,a)\inv\cdot \alpha \cdot\Omega(\xi,\eta,a)\cdot\Omega(\xi\eta,\zeta,a_4)\\
&\stackrel{\eqref{E 3.3}}{=}&
\Omega(\xi\eta,\zeta,a_4)\inv\cdot\Lambda_{\xi\eta}\inv\Lambda_\eta \Lambda_\xi(\alpha)\cdot \Omega(\xi\eta,\zeta,a_4)\\
&\stackrel{\eqref{E 3.3}}{=}&
\Lambda_{\xi\eta\zeta}\inv \Lambda_\zeta \Lambda_{\xi\eta}  \Lambda_{\xi\eta}\inv \Lambda_\eta \Lambda_\xi(\alpha)\\
&=&\Lambda_{\xi\eta\zeta}\inv \Lambda_\zeta\Lambda_\eta\Lambda_\xi(\alpha),
\end{eqnarray*}
and
\begin{eqnarray*}
&&\Omega(\xi,\eta\zeta,a_5)\inv \cdot \Lambda_\xi\inv[\Omega(\eta,\zeta,a_1)\inv]\cdot\alpha\cdot
\Lambda_\xi\inv[\Omega(\eta,\zeta_K,a_1)] \cdot\Omega(\xi,\eta\zeta,a_5)\\
&=& \Omega(\xi,\eta\zeta,a_5)\inv \cdot \Lambda_\xi\inv \left[\Omega(\eta,\zeta,a_1)\inv\cdot\Lambda_\xi(\alpha)\cdot \Omega(\eta,\zeta,a_1) \right]\cdot \Omega(\xi,\eta\zeta,a_5)\\
&\stackrel{\eqref{E 3.3}}{=}  & \Omega(\xi,\eta\zeta,a_5)\inv \cdot \Lambda_\xi\inv[\Lambda_{\eta\zeta}\inv\Lambda_\zeta \Lambda_\eta\Lambda_\xi(\alpha)]\cdot \Omega(\xi,\eta\zeta,a_5)\\
&\stackrel{\eqref{E 3.3}}{=}&\Lambda_{\xi\eta\zeta}\inv \Lambda_{\eta\zeta}\Lambda_\xi\Lambda_\xi\inv \Lambda_{\eta\zeta}\inv\Lambda_\zeta \Lambda_\eta\Lambda_\xi(\alpha) \\
&=&\Lambda_{\xi\eta\zeta}\inv\Lambda_\zeta\Lambda_\eta \Lambda_\xi(\alpha).
\end{eqnarray*}
It follows from these two equalities that the following arrow
\begin{align}\label{E def-Xi}
\begin{split}
\Xi(\xi,\eta,\zeta,a)
:&=\Omega(\xi,\eta,a)\cdot\Omega(\xi\eta,\zeta,a_4)
\cdot \Omega(\xi,\eta\zeta,a_5)\inv \cdot \Lambda_\xi\inv[\Omega(\eta,\zeta,a_1) \inv]
\end{split}
\end{align}
belongs to the center of $\Gamma_a$. Thus we get a smooth map
\[
\Xi:K^{[3]}\times A^0=K^1\times_{t,s}K^1\times_{t,s} K^1\times A^0\rto ZA^0\subseteq A^1.
\]
By \eqref{E 3.3} this map is invariant under the $\A$-action on $ZA^0$. Hence it gives rise to a smooth map (cf. Remark \ref{R smoothness-of-ZA})
\[
\Xi:K^{[3]}=K^1\times_{t,s}K^1\times_{t,s} K^1\rto Z_\A.
\]

It is obvious that
\begin{lemma}\label{L Xi=1-exist-extension}
When $\Xi(\cdot,\cdot,\cdot,a)\equiv 1_a$ for all $a\in A^0$, $(\Lambda,\Omega)$ satisfies the generalized cocycle condition \eqref{E generalized-cocycle-1}, and we get a topological product Lie groupoid fibration with band $\bar\Lambda$.
\end{lemma}

The morphism $\bar \Lambda:K^1\rto \overline{\sf{SAut}(\A)}$ induces a smooth $\K$-action on $Z_\A$ (see \S \ref{S appdndix-induced-K-action-on-ZA})
\[
\bar \Lambda:K^1\rto \mathrm{Aut}(Z_\A).
\]
Then we could define the groupoid cohomology $H^*_{\bar \Lambda}(\K,Z_\A)$ of $\K$ with coefficients in $Z_\A$ (see \S \ref{S appendix-groupoid-cohomology}).
\begin{theorem}\label{T Xi-cocycle}
$\Xi$ is a 3-cocycle in the cochain complex $C^*_{\bar \Lambda}(\K,Z_\A)$, hence represents a class $[\Xi]$ in $H^3_{\bar \Lambda}(\K,Z_\A)$. Moreover, this class $[\Xi]$ only depends on $\bar \Lambda$, not on $(\Lambda,\Omega)$.
\end{theorem}

We will prove this theorem in the appendix, \S \ref{S appendix-Xi}.

\begin{theorem}\label{T obs}
There is a topological product Lie groupoid fibration over $\K$ by $\A$ with band $\bar \Lambda:K^1\rto \overline{\sf{SAut}(\A)}$ if and only if $[\Xi]=0$ in $H^3_{\bar \Lambda}(\K,Z_\A)$.
\end{theorem}
\begin{proof}
The only if part follows from Lemma \ref{L Xi=1-exist-extension}. Now consider the if part. Suppose $[\Xi]=0$, then there is a $2$-cochian $c\in C^2_{\bar \Lambda}(\K,Z_\A)$ such that
\begin{align*}
\Xi(\xi_1,\xi_2,\xi_3)=dc(\xi_1,\xi_2,\xi_3)
=\Lambda_{\xi_1}\inv( c(\xi_2,\xi_3))\cdot c(\xi_1\xi_2,\xi_3)\inv\cdot c(\xi_1,\xi_2\xi_3)\cdot c(\xi_1,\xi_2)\inv.
\end{align*}
Set
\[
\Omega'(\xi_1,\xi_2,a):=c(\xi_1,\xi_2)(a)\cdot \Omega(\xi_1,\xi_2,a).
\]
Then $(\Lambda,\Omega')$ is a generalized cocycle. So via Theorem \ref{T S-Omega-def-G-SOmega} we get a topological trivial Lie groupoid fibration $\K\ltimes_{\Lambda,\Omega'}\A$ with band $\bar \Lambda$.
\end{proof}

We call $[\Xi]\in H^3_{\bar \Lambda}(\K,Z_\A)$ the {\bf obstruction class} of topological product Lie groupoid fibrations over $\K$ by $\A$ with band $\bar \Lambda$.

\begin{example}\label{Example 3}
As in Example \ref{Example 1} let $K$ and $A$ be two finite groups. Suppose $K$ acts smoothly on a smooth manifold $X$ from the right, and $A$ acts trivially on $X$. Consider the two Lie groupoids $[\bullet/K]$ and $A\ltimes X$. Let $\omega:K\rto \mathrm{Aut}(A)$ be a lifting of the homomorphism $\bar\omega:K\rto\mathrm{Out}(A)$. Then we get a homomorphism
$
\bar\Lambda: K\rto \overline{{\sf SAut}(A\ltimes X)}
$
by setting $\bar\Lambda(k)$ to be the image of the following $\Lambda_k\in\mathrm{SAut}^0(A\ltimes X)$:
\[
\Lambda_k:A\ltimes X\rto A\ltimes X,\qq \Lambda^0_k(x)=x\cdot k,\qq \Lambda^1_k(a,x)=(\omega(k)(a),\Lambda^0_k(x)).
\]
Obviously, $\Lambda_k$ forms a smooth lifting of $\bar\Lambda$. So we could study the existence of topological product fibration over $[\bullet/K]$ with fiber $A\ltimes X$. Since $\Lambda_k^0$ is the action map of the $K$-action on $X$, the corresponding cofactor $\Omega$ is determined by the default of $\omega$ being a homomorphism. As in Example \ref{Example 1} we denote it by $f:K\times K\rto A$. So $\mathrm{Inn}(f(k_1,k_2))=\omega(k_1k_2)\inv\circ\omega(k_2)\circ \omega(k_1)$. The obstruction class associated to $\bar\Lambda$ is a map $\Xi:K^3\times X\rto Z_{A\ltimes X}$. For simplicity we assume that $X$ is connected. So $Z_{A\ltimes X}=Z(A)$, the center of $A$; and $\Xi$ is constant on $X$, so it is a map $K^3\rto Z(A)$. This is the cocycle $c_{\bar\omega}$ determined by $(\omega,f)$. $\omega$ also determine a $K$-action on $Z(A)$, and $H^*_{\bar\Lambda}([\bullet/K],Z(A))=H^*_{\bar\omega}(K,Z(A))$. Therefore when $[c_{\bar\omega}]=0\in H^3_{\bar\omega}(K,Z(A))$, we have $[\Xi]=0\in H^3_{\bar\Lambda}([\bullet/K],Z_{A\ltimes X})$, and hence a topological product fibration over $[\bullet/K]$ with fiber $A\ltimes X$. This topological product fibration is $(K\ltimes_{(\omega,f)} A)\ltimes X\rto [\bullet/K]$ with $K\ltimes_{(\omega,f)} A$ being the group extension of $K$ by $A$ determined by $(\omega,f)$.
\end{example}

\section{Classification of locally topological product Lie groupoid fibrations}\label{S 5}

Now suppose we are given a morphism
$
\bar \Lambda:K^1\rto \overline{\sf{SAut}(\A)}.
$
In this section we classify all locally topological product Lie groupoid fibrations over $\K$ by $\A$ with band $\bar \Lambda$ when the obstruction class $[\Xi]=0\in H^3_{\bar \Lambda}(\K,Z_\A)$.

\subsection{Topological product Lie groupoid fibrations}
We first classify topological product Lie groupoid fibrations over $\K$ by $\A$ with band $\bar \Lambda$. By Corollary \ref{C G-cong-KSOA}, we only need to study topological product Lie groupoid fibrations constructed out of generalized cocycles.

Since $[\Xi]=0$, by Theorem \ref{T obs} every smooth lifting $\Lambda$ of $\bar \Lambda$ and the corresponding smooth $\Omega$ determines a topological product Lie groupoid fibration $\A\rtimes_{\Lambda,\Omega}\K$. However, there would exist different liftings $(\Lambda,\Omega)$ that yield isomorphic topological product Lie groupoid fibrations.

First note that, every natural transformation $\rho:\Lambda_1\Rto \Lambda_2$ in $\mathrm{SAut}^1(\A)$ has an inverse $\rho^{\circledast,-1}:\Lambda_1\inv\Rto \Lambda_2\inv$ under the horizontal composition, that is they satisfy the property that $\rho^{\circledast,-1}\circledast\rho= \rho\circledast\rho^{\circledast,-1}=1_{\sf id_A}:\sf id_A\Rto \sf id_A:A\rto A$ is the identity natural transformation. This inverse $\rho^{\circledast,-1}$ is given by
\begin{align}\label{E def-inverse-circledast}
\rho^{\circledast,-1}(a)=\left(\Lambda_1\inv \circ \rho\circ \Lambda_2\inv(a)\right)\inv,  \qq \forall a\in A^0.
\end{align}
A natural transformation $\rho$ also has an inverse $\rho^{\odot,-1}:\Lambda_2\Rto \Lambda_1$ with respect to the vertical composition. That is they satisfy the property that $\rho\odot\rho^{\odot,-1}=1_{\Lambda_1}:\Lambda_1\Rto \Lambda_1$ and $\rho^{\odot,-1}\odot\rho=1_{\Lambda_2}:\Lambda_2\Rto \Lambda_2$ are the identity natural transformations over $\Lambda_1$ and $\Lambda_2$ respectively. The inverse $\rho^{\odot,-1}$ is given by
\[
\rho^{\odot,-1}(a)=\rho(a)\inv, \qq \forall a\in A^0.
\]
\begin{defn}
Two generalized cocycles $(\Lambda,\Omega)$ and $(\Lambda',\Omega')$ are {\bf equivalent} if there is a smooth family of natural transformations $\rho_\xi:\Lambda_{\xi}\Rto \Lambda_{\xi}'$ parameterized by $K^1$ such that
\[
\Omega'(\xi,\eta)= \Omega(\xi,\eta)\odot [\rho_{\xi\eta}^{\circledast,-1}\circledast \rho_\eta\circledast \rho_\xi]
\]
and $\rho_{1_x}=u$, the unit map $u:A^0\rto A^1$.
\end{defn}

It is direct to see that
\begin{lemma}
Two topological product Lie groupoid fibrations $\A\rtimes_{\Lambda,\Omega}\K$ and $\A\rtimes_{\Lambda',\Omega'}\K$ are isomorphic if and only if $(\Lambda,\Omega)$ is equivalent to $(\Lambda',\Omega')$.
\end{lemma}

Denote the set of isomorphic classes of topological product Lie groupoid fibrations over $\K$ by $\A$ with band $\bar \Lambda$ by
$
\mathrm{Iso}(\K,\A,\bar \Lambda).
$
Denote the isomorphic class of a topological product Lie groupoid fibration $\A\rtimes_{\Lambda,\Omega}\K$ by $[\A\rtimes_{\Lambda,\Omega}\K]$.

\begin{theorem}\label{T classification-iso}
There is a 1-to-1 correspondence between $\mathrm{Iso}(\K,\A,\bar \Lambda)$ and $H^2_{\bar \Lambda}(\K,Z_\A)$.
\end{theorem}

\begin{proof}
By the above analysis, the isomorphic classes of topological product Lie groupoid fibrations over $\K$ with fiber $\A$ and band $\bar \Lambda$ correspond to the equivalence classes of generalized cocycles that lift $\bar \Lambda$.

Let $\A\rtimes_{\Lambda_0,\Omega_0}\K$ be the topological product Lie groupoid fibration corresponds to a fixed generalized cocycle $(\Lambda_0,\Omega_0)$ with $\pi\circ \Lambda_0=\bar \Lambda$. The existence of $\A\rtimes_{\Lambda_0,\Omega_0}\K$ is guaranteed by the assumption. We next construct a map from $H^2_{\bar \Lambda}(\K,Z_\A)$ to the set of equivalence classes of generalized cocycles by setting
\begin{align}\label{E def-c-to-cOmega}
[c]\mapsto [(\Lambda_0,c\Omega_0)],
\end{align}
where $c:K^{[2]}\rto Z_\A$ is a cocycle in $C^2_{\bar \Lambda}(\K,Z_\A)$, representing $[c]$. First of all, $(\Lambda_0,c\Omega_0)$ is still a generalized cocycle. So there is a topological product Lie groupoid fibration over $\K$ with fiber $\A$ corresponds to $(\Lambda_0,c\Omega_0)$, hence its band is $\bar \Lambda$.

The assignment \eqref{E def-c-to-cOmega} gives rise to a well-defined map. In fact, if $c$ and $c'$ represent the same cohomology class, there is a 1-cochian $\rho$ such that $d\rho=c'\cdot c\inv$; then $\rho$ gives rise to an equivalence between $(\Lambda_0,c\Omega)$ and $(\Lambda_0,c'\Omega_0)$. We next show that this map \eqref{E def-c-to-cOmega} is bijective.

For the surjectivity, suppose $(\Lambda,\Omega)$ is a generalized cocycle with $\Lambda$ being a smooth lifting of $\bar \Lambda$. Then since both $\Lambda$ and $\Lambda_0$ are smooth liftings of $\bar \Lambda$, there is a smooth family of natural transformations
$
\rho_\xi:\Lambda_{\xi}\Rto \Lambda_{0,\xi}
$
for all $\xi\in K^1$, and $\rho_{1_x}=u$ for every $x\in K^0$. Therefore for every pair of composable arrows $\xi_1,\xi_2$ in $K^1$, we have
\[
\rho_{\xi_1\xi_2}^{\circledast,-1}\circledast\rho_{\xi_2} \circledast\rho_{\xi_1}:\Lambda_{\xi_{12}}\inv \Lambda_{\xi_2}\Lambda_{\xi_1}\Rto \Lambda_{0,\xi_{12}}\inv \Lambda_{0,\xi_2}\Lambda_{0,\xi_1}.
\]
On the other hand, for $\xi_1,\xi_2$ we have
\[
\Omega(\xi_1,\xi_2):{\sf id}_\A\Rto \Lambda_{\xi_{12}}\inv \Lambda_{\xi_2}\Lambda_{\xi_1}, \qq
\Omega_0(\xi_1,\xi_2):{\sf id}_\A\Rto \Lambda_{0,\xi_{12}}\inv \Lambda_{0,\xi_2}\Lambda_{0,\xi_1}.
\]
Therefore we obtain a natural transformation
\[
\Omega'(\xi_1,\xi_2):= \Omega(\xi_1,\xi_2)\odot \left(\rho_{\xi_1\xi_2}^{\circledast,-1}\circledast\rho_{\xi_2} \circledast\rho_{\xi_1}\right): {\sf id_A}\Rto \Lambda_{0,\xi_{12}}\inv \Lambda_{0,\xi_2}\Lambda_{0,\xi_1}.
\]
This means that $(\Lambda,\Omega)$ is equivalent to $(\Lambda_0,\Omega')$. The two natural transformations $\Omega'(\xi_1,\xi_2)$ and $\Omega_0(\xi_1,\xi_2)$ induce a
\[
c(\xi_1,\xi_2):=\Omega'(\xi_1,\xi_2)\odot \Omega_0(\xi_1,\xi_2)^{\odot,-1}: {\sf id}_\A\Rto {\sf id}_\A.
\]
By item (3) of Proposition \ref{P SAut}, $c(\xi_1,\xi_2)\in Z_\A$. Therefore $c\in C_{\bar \Lambda}^2(\K,Z_\A)$. The fact that both $(\Lambda_0,\Omega_0)$ and $(\Lambda_0,\Omega')$ are generalized cocycle implies that $c$ is a cocycle. Moreover $\Omega'=c\odot\Omega_0=c\cdot\Omega_0$. Therefore the map \eqref{E def-c-to-cOmega} is surjective.

Finally, we consider the injectivity. If $(\Lambda_0,c\Omega_0)$ is equivalent to $(\Lambda_0,c'\Omega_0)$. Then there is a smooth family of natural transformations $\rho_{\xi}:\Lambda_{0,\xi}\Rto \Lambda_{0,\xi}$ such that for every composable pair of arrows $\xi_1,\xi_2\in K^1$,
\[
(c\Omega_0)(\xi_1,\xi_2)=(c'\Omega_0)(\xi_1,\xi_2) \odot (\rho_{\xi_1\xi_2}^{\circledast,-1}\circledast\rho_{\xi_2} \circledast\rho_{\xi_1}).
\]
This gives rise to a $\rho\in C^1_{\bar \Lambda}(\K,Z_\A)$, since the isotropy group of $\Lambda_0$ in $\sf SAut(A)$ is also $Z_\A$ by Proposition \ref{P SAut}. Moreover
$
c=c'\cdot d\rho.
$
Therefore the map \eqref{E def-c-to-cOmega} is injective. This finishes the proof.
\end{proof}

\subsection{Classifying locally topological product Lie groupoid fibrations}

The morphism $\bar\Lambda:K^1\rto\overline{\sf SAut(A)}$ induces a morphism
\begin{align}\label{E induced-Lambda-U}
\bar \Lambda_{\mc U}=\bar \Lambda\circ q^1_{\mc U}:\K[\mc U]^1\rto \overline{\sf{SAut}(\A)}.
\end{align}
for every open cover $\mc U$ of $K^0$ via the refinement morphism $\sq:\K[\mc U]\rto\K$. The obstruction class associated to $\bar \Lambda_{\mc U}$ is the image of the obstruction class $[\Xi]$ associated to $\bar \Lambda$ under the homomorphism
\[
\msf q^*_{\mc U}:H^*_{\bar \Lambda}(\K,Z_\A)\rto H^3_{\bar \Lambda_{\mc U}}(\K[\mc U],Z_\A),
\]
hence also vanishes by the assumption $[\Xi]=0$. Therefore, for every open cover $\mc U$ of $K^0$, there exists topological product Lie groupoid fibration over $\K[\mc U]$ with fiber $\A$ and band $\bar\Lambda_{\mc U}$.

In this subsection, w.r.t. equivalence of locally topological product Lie groupoid fibrations we classify all locally topological product Lie groupoid fibrations with band induced by $\bar\Lambda$ via \eqref{E induced-Lambda-U}.
By the definition of locally topological product Lie groupoid fibrations, every locally topological product Lie groupoid fibration $\G\xrightarrow{\phi}\K$ has a refinement $\G[\phi^*\mc U]\rto\K[\mc U]$ that is isomorphic to a topological product Lie groupoid fibration; and by the Definition \ref{D equivalence-extension} of equivalence of locally topological product Lie groupoid fibrations, $\G\xrightarrow{\phi}\K$ is equivalent to $\G[\phi^*\mc U]\rto\K[\mc U]$. However, there would exist topological product Lie groupoid fibrations over refinements of $\K$ with fiber $\A$ that are not refinements of locally topological product Lie groupoid fibrations $\G\xrightarrow{\phi}\K$ over $\K$ with fiber $\A$. So we consider the following set
\[
\widetilde{\mathrm{LTPLGF}}(\K,\A,\bar \Lambda):=\left\{
\A\rtimes_{\Lambda_{\mc U},\Omega_{\mc U}}\K[\mc U] \left|\substack{\mc U \text{ is an open cover of } K^0,\\ \Lambda_{\mc U} \text{ is a smooth lifting of }\bar \Lambda_{\mc U},\text{ and}\\
(\Lambda_{\mc U},\Omega_{\mc U}) \text{ is a generalized cocycle}}\right. \right\}.
\]
Denote by
$
\mathrm{LTPLGF}(\K,\A,\bar \Lambda)$ the quotient
of this set
by the equivalence relation between locally topological product Lie groupoid fibrations given in Definition \ref{D equivalence-extension}.

First of all we have the following two simple lemmas.
\begin{lemma}\label{L canonical-iso-on-pullbacks}
Every generalized cocycle $(\Lambda,\Omega)$ over $\K$ pulls back to a generalized cocycle $(\Lambda_{\mc U},\Omega_{\mc U}):=\sq_{\mc U}^*(\Lambda,\Omega)$ over $\K[\mc U]$ with
\[
\Lambda_{\mc U}:=\Lambda\circ q^1_{\mc U}:\K[\mc U]^1\rto\mathrm{SAut}(\A)
\]
and
\[
\Omega_{\mc U}:=\Omega\circ (q^1_{\mc U},q^1_{\mc U},id_{A^0}):\K[\mc U]^1\times_{t,s}\K[\mc U]^1\times A^0\rto A^1.
\]
Moreover, there is a canonical isomorphism
\[
\A\rtimes_{\Lambda_{\mc U},\Omega_{\mc U}}\K[\mc U]\cong \sq^*_{\mc U}(\A\rtimes_{\Lambda,\Omega}\K).
\]
\end{lemma}

\begin{lemma}\label{L canonical-iso-on-refinements}
$\A\rtimes_{\Lambda,\Omega}\K[\mc U]$ is equivalent to $\A\ltimes_{\Lambda',\Omega'}\K[\mc U']$ if and only if there is a third open cover $\mc W$ of $K^0$ refining both $\mc U$ and $\mc U'$, such that the pullback generalized cocycles $(\Lambda_{\mc W},\Omega_{\mc W})$ and $(\Lambda_{\mc W}',\Omega_{\mc W}')$ are equivalent.
\end{lemma}

For a fixed $\mc U$ we have the set of isomorphic classes of topological product Lie groupoid fibrations $\mathrm{Iso}(\K[\mc U],\A,\bar \Lambda_{\mc U})$. If $\mc W$ is another open cover that refines $\mc U$ via $\iota_{\mc W,\mc U}:\mc W\rto\mc U$, then we have the commutative diagram of morphisms
\begin{align}\label{E 5.1}
\begin{split}
\xymatrix{
\K[\mc W]\ar[rr]^-{\iota_{\mc W,\mc U}} \ar[drr]_-{\sq_{\mc W}} &&
\K[\mc U]\ar[d]^-{\sq_{\mc U}} \\
&&\K  }
\end{split}
\end{align}
and
\begin{align}\label{E 5.2}
\bar \Lambda_{\mc U}\circ  \iota_{\mc W,\mc U}=\bar \Lambda_{\mc W}.
\end{align}
Note that, the commutative diagram \eqref{E 5.1} and the equality \eqref{E 5.2} do not depend on the choice of the refine map $\iota_{\mc W,\mc U}:\mc W\rto\mc U$. So in the following, we do not specify the refinement map $\iota_{\mc W,\mc U}$.

The refinement morphism $\iota_{\mc W,\mc U}$ pulls back a generalized cocycle $(\Lambda_{\mc U},\Omega_{\mc U})$ over $\K[\mc U]$ to a generalized cocycle $(\iota^*_{\mc W,\mc U}\Lambda_{\mc U},\iota^*_{\mc W,\mc U}\Omega_{\mc U})$ over $\K[\mc W]$. Moreover similar as Lemma \ref{L canonical-iso-on-pullbacks} we have
\[
\iota^*_{\mc W,\mc U}(\A\rtimes_{\Lambda_{\mc U},\Omega_{\mc U}}\K[\mc U])\cong\A\rtimes_{\iota^*_{\mc W,\mc U}\Lambda_{\mc U},\iota^*_{\mc W,\mc U}\Omega_{\mc U}}\K[\mc W].
\]
Therefore we have a refinement map
\begin{align*}
\iota_{\mc W,\mc U}^*:\mathrm{Iso}(\K[\mc U],\A,\bar \Lambda_{\mc U})
\rto
\mathrm{Iso}(\K[\mc W],\A,\bar \Lambda_{\mc W}),\qq
[\A\ltimes_{\Lambda_{\mc U},\Omega_{\mc U}}\K[\mc U]]&\mapsto [\A\ltimes_{\iota^*_{\mc W,\mc U}\Lambda_{\mc U},\iota^*_{\mc W,\mc U}\Omega_{\mc U}}\K[\mc W]].
\end{align*}

By taking all open covers of $K^0$, the collection of isomorphic classes $\mathrm{Iso}(\K[\mc U],\A,\bar \Lambda_{\mc U})$ and of refinement maps $\iota_{\mc W,\mc U}^*:\mathrm{Iso}(\K[\mc U],\A,\bar \Lambda_{\mc U}) \rto \mathrm{Iso}(\K[\mc W],\A,\bar \Lambda_{\mc W})$ forms a direct system. Thus we can form the direct limit set
\[
\lim_{\mc U} \mathrm{Iso}(\K[\mc U],\A,\bar \Lambda_{\mc U}).
\]
Then by the definition of equivalence of locally topological product Lie groupoid fibrations in Definition \ref{D equivalence-extension} and the definition of direct limit we get the following result.
\begin{prop}
There is a 1-to-1 map
\[
\mathrm{LTPLGF}(\K,\A,\bar \Lambda)\xymatrix{
\ar[rr]^-{\text{\em 1-to-1}} && \ar[ll]}\lim_{\mc U} \mathrm{Iso}(\K[\mc U],\A,\bar \Lambda_{\mc U}),
\]
induced by identity map on the set $\widetilde{\mathrm{LTPLGF}}(\K,\A,\bar \Lambda)$.
\end{prop}

We next consider the cohomology groups $H^*_{\bar \Lambda_{\mc U}}(\K[\mc U],Z_\A)$ for open covers of $K^0$. For a refinement $\iota_{\mc W,\mc U}:\mc W\rto \mc U$ we also have a group homomorphism
\[
\iota^*_{\mc W,\mc U}: H^*_{\bar \Lambda_{\mc U}}(\K[\mc U],Z_\A)\rto H^*_{\bar \Lambda_{\mc W}}(\K[\mc W],Z_\A).
\]
As the construction of $\check{\mathrm{C}}$ech cohomology for manifolds in \cite{Warner1983} and for simplicial spaces in \cite{Tu2006}, this homomorphism $\iota^*_{\mc W,\mc U}$ does not depend on the choice of the explicit refinement map $\iota_{\mc W,\mc U}:\mc W\rto\mc U$. Then the collection of these cohomology groups and these group homomorphisms also forms a direct system, and we have the direct limit
\[
\lim_{\mc U} H^*_{\bar \Lambda_{\mc U}}(\K[\mc U],Z_\A).
\]

Now fix a reference smooth lifting $\Lambda_0:K^1\rto\mathrm{SAut}(\A)$ of $\bar \Lambda$, and a smooth cofactor $\Omega_0$ such that $(\Lambda_0,\Omega_0)$ is a generalized cocycle. The generalized cocycle $(\Lambda_0,\Omega_0)$ pulls back to generalized cocycle $\sq^*_{\mc U}(\Lambda_0,\Omega_0)=(\Lambda_{0,\mc U},\Omega_{0,\mc U})$ over $\K[\mc U]$ for every open cover $\mc U$ of $K^0$. By Theorem \ref{T classification-iso} we have the identification
\begin{align*}
\Phi_{\mc U}: H^2_{\bar \Lambda_{\mc U}}(\K[\mc U],Z_\A)
&
\xymatrix{~\ar[rr]^-{\text{1-to-1}} && ~\ar[ll]}
\mathrm{Iso}(\K[\mc U],\A,\bar \Lambda_{\mc U}), \qq
[c]
\mapsto
[\K\ltimes_{\Lambda_{0,\mc U}, c\Omega_{0,\mc U}}\A]
\end{align*}
for every open cover $\mc U$ of $K^0$.

\begin{theorem}\label{T classification-equiv}
Two topological product Lie groupoid fibrations $\K[\mc U]\ltimes_{\Lambda_{0,\mc U},c\Omega_{0,\mc U}}\A$ and $\K[\mc V]\ltimes_{\Lambda_{0,\mc V},c'\Omega_{0,\mc V}}\A$ are equivalent, if there is a third cover $\mc W$ of $K^0$ that refines both $\mc U$ and $\mc V$ via $\iota_{\mc W,\mc U}$ and $\iota_{\mc W,\mc V}$ such that
\[
\iota^*_{\mc W,\mc U}([c])=\iota^*_{\mc W,\mc V}([c']).
\]
Moreover, we have a 1-to-1 map
\[
\mathrm{LTPLGF}(\K,\A,\bar \Lambda)\xymatrix{
\ar[rr]^-{\text{\em 1-to-1}}&&
\ar[ll]}\lim_{\mc U} H^2_{\bar \Lambda_{\mc U}}(\K[\mc U],Z_\A).
\]
\end{theorem}
\begin{proof}
The first assertion follows from the definition and is a restatement of Lemma \ref{L canonical-iso-on-refinements}. To prove the second assertion, we only have to show that the following diagram
\[
\xymatrix{
H^2_{\bar \Lambda_{\mc U}}(\K[\mc U],Z_\A)
\ar[rr]^-{\Phi_{\mc U}} \ar[d]_-{\iota^*_{\mc W,\mc U}}&&
\mathrm{Iso}(\K[\mc U],\A,\bar \Lambda_{\mc U}) \ar[d]_-{\iota^*_{\mc W,\mc U}}
\\
H^2_{\bar \Lambda_{\mc W}}(\K[\mc W],Z_\A)
\ar[rr]^-{\Phi_{\mc W}}&&
\mathrm{Iso}(\K[\mc W],\A,\bar \Lambda_{\mc W})
}
\]
is commutative. This follows from the fact that all reference generalized cocycles $(\Lambda_{0,\mc U},\Omega_{0,\mc U})$ are pull backs of the fixed generalized cocycle $(\Lambda_0,\Omega_0)$ via refinements. This finishes the proof.
\end{proof}

\begin{example}\label{Example 4}
Consider the Example \ref{Example 3}. When $c_{\bar\omega}=0$, we have topological product fiber bundles with fiber $A\ltimes X$ over $[\bullet/K]$ whose band are $\bar\Lambda$. In this circumstance, since $H^2_{\bar\Lambda}([\bullet/K],Z_{A\ltimes X})= H^2_{\bar\omega}(K,Z(A))$, there are exactly $H^2_{\bar\omega}(K,Z(A))$ non-isomorphic topological product fiber bundles over $[\bullet/K]$ with fiber $A\ltimes X$ and band $\bar\Lambda$.

Moreover, since $[\bullet/K]$ has no nontrivial refinements, all locally topological product Lie groupoid fibrations over $[\bullet/K]$ with fiber $A\ltimes X$ are topologically product Lie groupoid fibrations. Hence there are also exactly $H^2_{\bar\omega}(K,Z(A))$ inequivalence locally topological product fiber bundles over $[\bullet/K]$ with fiber $A\ltimes X$ and band induced by $\bar\Lambda$. That is
\[
\mathrm{LTPLGF}(\K,\A,\bar \Lambda)\xymatrix{
\ar[rr]^-{\text{ 1-to-1}}&&
\ar[ll]}H^2_{\bar \omega}([\bullet/K],Z(A))\xymatrix{
\ar[rr]^-{\text{ 1-to-1}}&&
\ar[ll]}
\mathrm{Iso}([\bullet/K],A\ltimes X,\bar \Lambda).
\]
\end{example}

\section{Symplectic structure over locally topological product Lie groupoid fibrations}\label{S 6}

Suppose that both $\K$ and $\A$ are orbifold groupoids. Then the topological product Lie groupoid fibration $\G=\A\rtimes_{\Lambda,\Omega}\K$ associated to a generalized cocycle $(\Lambda,\Omega)$ is also an orbifold groupoid. In this section we show that under some appropriate conditions, there is a natural symplectic structure over $\G=\A\rtimes_{\Lambda,\Omega}\K$ when both $\K$ and $\A$ are symplectic orbifold groupoids.

\begin{theorem}\label{T induced-symplectic-on-KltimesA}
Let $\omega_K$ and $\omega_A$ be the symplectic forms over the orbifold groupoids $\K$ and $\A$ respectively. If $\Lambda:K^1\rto \mathrm{SAut}^0(\A)$ preserves the symplectic structure over $\A$ and is locally constant, then the symplectic form
\[
\omega_G:=\omega_A\times\omega_K=pr_1^*\omega_A\wedge pr_2^*\omega_K
\]
is a symplectic form over $\G=\A\rtimes_{\Lambda,\Omega}\K$, where $pr_i$ are the projections from $G^0=A^0\times K^0$ to its two factors for $i=1,2$.
\end{theorem}

\begin{proof}
By the definition of symplectic form over orbifold groupoid we need to show that for two arbitrary tangent vector fields $X_G:=(X_A,X_K),Y_G:=(Y_A,Y_K)\in \Gamma(TG^1)=\Gamma(T(A^1\times K^1))$ we have
\[
s^*_G\omega_G(X_G,Y_G)=t^*_G\omega_G(X_G,Y_G),
\]
where $s_G$ and $t_G$ are the source and target maps of $\G$ defined in \eqref{E def-s-in-K-ltimes-A} and \eqref{E def-t-in-K-ltimes-A}.

First of all, by the definition of source map in \eqref{E def-s-in-K-ltimes-A} we have
\begin{align*}
s^*_G\omega_G(X_G,Y_G)&=\omega_G(s_{G,*}X_G,s_{G,*}Y_G) \\
&=\omega_G\left((s_{A,*}X_A,s_{A,*}Y_A),(s_{K,*}X_K,s_{K,*}Y_K)\right)\\
&=\omega_A(s_{A,*}X_A,s_{A,*}Y_A)\cdot \omega_K(s_{K,*}X_K,s_{K,*}Y_K).
\end{align*}
where $s_K$ and $s_A$ are the source maps of $\K$ and $\A$ respectively. Next we consider $t^*_G\omega_G(X_G,Y_G)$. By the definition of pulling back of forms we have
\begin{align*}
t^*_G\omega_G(X_G,Y_G)&=\omega_G(t_{G,*}X_G,t_{G,*}Y_G).
\end{align*}
However $t_{G,*}$ is much more complicate by the definition in \eqref{E def-t-in-K-ltimes-A}. We have
\[
t_G(\alpha,\xi)=(\Lambda_\xi\circ t_A(\alpha),t_K(\xi)).
\]
Write $\Lambda_\xi\circ t_A(\alpha)$ as $f(\alpha,\xi)$. Then
\[
t_{G,*}(X_G)=(\frac{\partial f}{\partial\alpha}X_A ,t_{K,*}X_K+\frac{\partial f}{\partial \xi}X_K).
\]
However, by the assumption that $\Lambda$ is locally constant we have $\frac{\partial f}{\partial \xi}=0$. Therefore
$
t_{G,*}(X_G)=(\frac{\partial f}{\partial\alpha}X_A, t_{K,*}X_K).
$
Similarly $t_{G,*}(Y_G)=(\frac{\partial f}{\partial\alpha}Y_A, t_{K,*}Y_K)$. So we have
\begin{align*}
t^*_G\omega_G(X_G,Y_G) &=\omega_G\left((\frac{\partial f}{\partial\alpha}X_A, t_{K,*}X_K),(\frac{\partial f}{\partial\alpha}Y_A, t_{K,*}Y_K)\right)\\
&=\omega_A(\frac{\partial f}{\partial\alpha}X_A,\frac{\partial f}{\partial\alpha}Y_A)\cdot\omega_K(t_{K,*}X_K ,t_{K,*}Y_K )\\
&=\omega_A(\frac{\partial f}{\partial\alpha}X_A,\frac{\partial f}{\partial\alpha}Y_A)\cdot\omega_K(s_{K,*}X_K ,s_{K,*}Y_K ),
\end{align*}
where for the last equality we have used the assumption that $(\K,\omega_K)$ is a symplectic orbifold groupoid. On the other hand we have
$
\frac{\partial f}{\partial\alpha}X_A=d(\Lambda_\xi)\circ t_{A,*}(X_A),
\frac{\partial f}{\partial\alpha_A}Y_A=d(\Lambda_\xi)\circ t_{A,*}(Y_A).
$
Therefore
\begin{align*}
\omega_A(\frac{\partial f}{\partial\alpha}X_A,\frac{\partial f}{\partial\alpha}Y_A)
&=\omega_A(d(\Lambda_\xi)\circ t_{A,*}(X_A),d(\Lambda_\xi)\circ t_{A,*}(Y_A))\\
&=\Lambda_\xi^*\omega_A(t_{A,*}X_A,t_{A,*}Y_A)\\
&=\omega_A(t_{A,*}X_A,t_{A,*}Y_A)\\
&=\omega_A(s_{A,*}X_A,s_{A,*}Y_A),
\end{align*}
where we have use the assumption that $\Lambda$ preserves the symplectic form over $\A$ for the third equality and the assumption that $\omega_A$ is a symplectic form over $\A$ for the last equality. Consequently,
\[
t^*_G\omega_G(X_G,Y_G)=s^*_G\omega_G(X_G,Y_G).
\]
Therefore $\omega_G$ is a symplectic form over $\G=\A\rtimes_{\Lambda,\Omega}\K$.
 \end{proof}

Obviously, the restriction of $\omega_G$ over the fiber of $\A\rtimes_{\Lambda,\Omega}\K\rto\K$, i.e. the fiber of the kernel, is isomorphic to $(\A,\omega_A)$.

It is direct to see that the existence of symplectic forms is a Morita equivalence invariant (see for example \cite[Proposition 7.3]{Du-Chen-Wang2018}), that is if $\f:\G\rto \sH$ is an equivalence between orbifold groupoids, then a symplectic form over $\G$ naturally induces a symplectic form over $\sH$ and vice versa.

\begin{coro}\label{C induced-symplectic-on-extension}
Let $\G\xrightarrow{\phi}\K$ be a locally topological product Lie groupoid fibration by $\A$ over $\K$ with $\G[\phi^*\mc U]\rto\K[\mc U]$ being a topological product Lie groupoid fibration by $\A$ that refines $\G\rto\K$. Suppose that both $\A$ and $\K$ are symplectic orbifold groupoids. Then if the pre-action map of $\G[\phi^*\mc U]\rto\K[\mc U]$ satisfies the assumption in Theorem \ref{T induced-symplectic-on-KltimesA}, there is a natural symplectic form over $\G$.
\end{coro}

\begin{example}
For example, for a finite group $A$ consider an \'etale $A$-gerbe $\G$
\begin{align}\label{E A-gerbe}
\begin{split}
\xymatrix{
\ker j \ar@{^(->}[r]^-i \ar@<0.4ex>[d]\ar@<-0.4ex>[d]&
G^1 \ar[r]^-j \ar@<0.4ex>[d]\ar@<-0.4ex>[d] &
K^1 \ar@<0.4ex>[d]\ar@<-0.4ex>[d]\\
K^0 \ar@{=}[r] &
K^0 \ar@{=}[r] &
K^0}
\end{split}
\end{align}
over a symplectic orbifold groupoid $(\K,\omega_K)$, where $\ker j$ is a locally trivial $A$-bundle over $K^0$. By the commutative diagram \eqref{E A-gerbe} we see that $\omega_K$ on $K^0$ is also compatible with the source and target maps of $\G$. Hence we get the induced symplectic structure over $\G$.

On the other hand, $\G$ is a locally topological product Lie groupoid fibration by $[\bullet/A]$ over $\K$, and $G^0=\bullet\times K^0=K^0$. The symplectic form over $[\bullet/A]$ is trivial.
By taking an appropriate open cover $\mc U$ of $K^0$ we get the refinement (see for example \cite{Laurent-Gengoux-Stienon-Xu2009,Tang-Tseng2014a})
\[
\xymatrix{
A\times\K[\mc U]^0\ar@{^(->}[r]^-i \ar@<0.4ex>[d]\ar@<-0.4ex>[d]&
A\times\K[\mc U]^1\ar[r]^-j \ar@<0.4ex>[d]\ar@<-0.4ex>[d] &
\K[\mc U]^1 \ar@<0.4ex>[d]\ar@<-0.4ex>[d]\\
\K[\mc U]^0 \ar@{=}[r] &
\K[\mc U]^0 \ar@{=}[r] &
\K[\mc U]^0.}
\]
Then one can see the pre-action map $\Lambda_{\mc U}$ for the refinement fibration is locally constant. Since the symplectic form over $[\bullet/A]$ is trivial, $\Lambda_{\mc U}$ preserves the symplectic form on $[\bullet/A]$. Therefore by the previous corollary the $A$-gerbe $\G$ over the symplectic orbifold groupoid $\K$ also has a naturally induced symplectic form, which is exactly the one induced from $\omega_K$ via $G^0=K^0$.
\end{example}

\section{Induced fibration over morphism groupoid}\label{S 7}

Let $\G=\A\rtimes_{\Lambda,\Omega}\K\xrightarrow{\phi}\K$ be a topological product Lie groupoid fibration corresponding to a generalized cocycle $(\Lambda,\Omega)$ over $\K$. $\phi$ induces a map on coarse space
$
\overline{\phi}:\overline\G\rto\overline\K.
$
It is a fibration of topological space with fiber being the coarse space $\overline\A$ of $\A$. $\overline\phi$ induces a group homomorphism over homology groups
$
\overline\phi_*:H_*(\overline\G,\integer)\rto H_*(\overline\K,\integer).
$
We call a homology classes $\beta\in H_*(\overline\G,\integer)$ a {\bf fiber class} if
$
\overline\phi_*(\beta)=0\in H_*(\overline\K,\integer).
$
A fiber class gives rise to a homology class in $H_*(\overline\A,\integer)$, which is mapped to $\beta$ via the induced homomorphism of the inclusion of $\overline\A$ into $\overline\G$ as a fiber. We denote this class in $H_*(\overline\A,\integer)$ still by $\beta$.

When $\A$ and $\K$ are both symplectic orbifold groupoid, under the assumption in Theorem \ref{T induced-symplectic-on-KltimesA}, we get a symplectic structure over $\G$. Then we could study the orbifold Gromov--Witten theory of $\G$. The orbifold Gromov--Witten theory of $\G$ has a parameter being the homology classes in $H_2(\overline\G,\integer)$. It is natural to expect that orbifold Gromov--Witten theory of $\G$ w.r.t. to a degree 2 fiber class is determined by the orbifold Gromov--Witten theory of $\A$ and the geometry of the Lie groupoid fibration $\G\rto \K$. The first step to study the orbifold Gromov--Witten theory of $\G$ is to study the morphism groupoid of generalized morphisms from orbifold Riemannian surfaces to $\G$. This was carried out for general groupoids in  \cite{Chen-Du-WangR2019} by R. Wang and the first two authors. In this section we study the relation between the morphism groupoid of fiber class generalized morphisms from an orbifold groupoid $\sH=(H^1\rrto H^0)$ to $\G$ and the morphism groupoid of generalized morphisms from $\sH$ to $\A$.

\subsection{Morphism groupoids}\label{S generalized-morphism-to-A}
We first recall the definition of morphism groupoids from \cite{Chen-Du-WangR2019}. We only consider the morphism groupoid of generalized morphisms between Lie groupoids.
\begin{defn}
A {\bf generalized morphism} $(\U,\psi,\su):\sH\rightharpoonup\A$ consists of a triple $(\U,\psi,\su)$, where $\psi:\U\rto\sH$ is an equivalence of Lie groupoids and $\su:\U\rto \A$ is a strict morphism between Lie groupoids.
\end{defn}

\begin{defn}
Let $(\U_1,\psi_1,\su_1),(\U_2,\psi_2,\su_2):\sH\rightharpoonup\A$ be two generalized morphisms. An {\bf arrow} $\alpha:(\U_1,\psi_1,\su_1)\rto (\U_2,\psi_2,\su_2)$ is a natural transformation that fits into the following diagram
\[
\begin{tikzpicture}
\def \x{3}
\def \y{1}
\node (A00) at (0,0)       {$\sH$};
\node (A10) at (\x,0)      {$\U_1\times_\sH\U_2$};
\node (A11) at (\x,\y)     {$\U_1$};
\node (A1-1) at (\x,-1*\y) {$\U_2$};
\node (A30) at (3*\x,0)    {$\A$.};
\node at (1.8*\x,0) {$\Downarrow \alpha$};
\path (A00) edge [<-] node [auto] {$\scriptstyle{\psi_1}$} (A11);
\path (A00) edge [<-] node [auto,swap] {$\scriptstyle{\psi_2}$} (A1-1);
\path (A11) edge [->] node [auto] {$\scriptstyle{\su_1}$} (A30);
\path (A1-1) edge [->] node [auto,swap] {$\scriptstyle{\su_2}$} (A30);
\path (A10) edge [->] node [auto] {$\scriptstyle{\pi_1}$} (A11);
\path (A10) edge [->] node [auto,swap] {$\scriptstyle{\pi_2}$} (A1-1);
\end{tikzpicture}
\]
where $\U_1\times_\sH\U_2=\U_1\times_{\psi_1,\sH,\psi_2}\U_2$ is the fiber product (cf. \cite{Adem-Leida-Ruan2007,Chen-Du-WangR2019}) of $\psi_1:\U_1\rto\sH$ and $\psi_2:\U_2\rto \sH$.
\end{defn}

Denote by $\mathrm{Mor}^0(\sH,\A)$ the space of generalized morphism from $\sH$ to $\A$ and $\mathrm{Mor}^1(\sH,\A)$ the space of arrows between generalized morphisms in $\mathrm{Mor}^0(\sH,\A)$. Then there is a (vertical) composition ``$\bullet$'' between arrows that makes
\[
\sf Mor(H,A):=(\mathrm{Mor}^1(\sH,\A)\rrto \mathrm{Mor}^0(\sH,\A))
\]
into a groupoid\footnote{Here we do not consider the smooth structure over it, so we only view it as a set level groupoid}. In fact this composition is constructed from the vertical composition of natural transformations. See \cite[Theorem 3.7]{Chen-Du-WangR2019} for the explicit construction.

Similarly we have the morphism groupoid $\sf Mor(H,G)$. This groupoid has several interesting subgroupoids. Given a generalized morphism $(\U,\psi,\su):\sH\rightharpoonup \G$, there are induced continuous maps over coarse spaces
\[
\xymatrix{\overline\sH &
\overline\U \ar[l]_-{\overline\psi}\ar[l]^-\cong\ar[r]^-{\overline\su} & \overline\G.}
\]
We say that this generalized morphism represents a fiber class if
$
\overline\phi\circ\overline\su\circ \overline\psi\inv=\{pt\}\in\overline\K.
$
Denote by $\mathrm{Mor}^0(\sH,\G)_F$ the space of all fiber class generalized morphisms in $\mathrm{Mor}^0(\sH,\G)$. Then we get a sub-groupoid
\[
{\sf Mor(H,G)}_F={\sf Mor(H,G)}|_{\mathrm{Mor}^0(\sH,\G)_F}
\]
of $\sf Mor(H,G)$.

Next we assume that the coarse space $\overline\sH$ of $\sH$ has a fundamental class $[\overline\sH]\in H_*(\overline\sH,\integer)$; for example when $\sH$ is a compact oriented orbifold groupoid, the assumption holds. Take a fiber class $\beta\in H_*(\overline\G,\integer)$, which determined a class $\beta\in H_*(\overline\A,\integer)$. Let $\mathrm{Mor}^0(\sH,\G)_\beta$ be the space of all fiber class generalized morphisms in $\mathrm{Mor}^0(\sH,\G)_F$ such that $\overline\su_*\circ \overline\psi_*\inv([\overline\sH])=\beta$. Then we have another sub-groupoid
\[
{\sf Mor(H,G)}_\beta={\sf Mor(H,G)}|_{\mathrm{Mor}^0(\sH,\G)_\beta}.
\]
of $\sf Mor(H,G)$. Similarly, we have a sub-groupoid $\sf Mor(H,A)_\beta$ of $\sf Mor(H,A)$.

\subsection{Induced groupoid fibration}\label{S subs-induced-grp-extension-GW}
In this subsection we show that the generalized cocycle $(\Lambda,\Omega)$ induces a generalized cocycle  $(\tilde \Lambda,\tilde \Omega)$ with $\tilde \Lambda:K^1\rto\mathrm{SAut}^0(\sf Mor(H,A))$ and a fibration\footnote{Here note that the morphism groupoids are only set level groupoids, so the groupoid fibrations we get are only set level groupoid fibrations. We will deal with the smooth structure over morphism groupoids in \cite{Chen-Du-Ono}.} over $\K$ with fiber $\sf Mor(\sH,\A)$.

We first define $\tilde \Lambda$. For every $\xi\in K^1$, we have an automorphism $\Lambda_\xi:\A\rto \A$. It induces a morphism
\[
\tilde \Lambda_\xi:\sf Mor(H,A)\rto\sf Mor(H,A)
\qq
\mathrm{by}
\qq
(\U,\psi,\su)\mapsto (\U,\psi,\Lambda_\xi\circ\su)\qq\mathrm{and}\qq
\alpha\mapsto \Lambda_\xi\circ\alpha.
\]
It is direct to see that this $\tilde \Lambda_\xi$ is an automorphism of $\sf Mor(H,A)$. This gives us the $\tilde \Lambda$.

Next we define $\tilde \Omega$. It will assign each pair $(\xi,\eta)\in K^{[2]}=K^1\times_{t,s}K^1$ a natural transformation $\tilde \Omega(\xi,\eta)$ from $\sf id_{ Mor(H,A)}$ to $\tilde \Lambda_{\xi\eta}\inv \circ \tilde \Lambda_\eta\circ \tilde \Lambda_\xi$. We next write down the explicit expression of $\tilde \Omega$. First of all it is a map
\[
\tilde \Omega:K^1\times_{t,s}K^1\times \mathrm{Mor}^0(\sH,\A) \rto \mathrm{Mor}^1(\sH,\A) .
\]
Let
$
(\U_1,\psi_1,\su_1)\xrightarrow{\alpha_{\sf Mor}}(\U_2,\psi_2,\su_2)
$ 
be an arrow in $\mathrm{Mor}(\sH,\A)^1$. For $i=1,2$,
\[
\tilde \Lambda_{\xi\eta}\inv \circ \tilde \Lambda_\eta\circ \tilde \Lambda_\xi(\U_i,\psi_i,\su_i)=(\U_i,\psi_i, \Lambda_{\xi\eta}\inv\circ\Lambda_\eta\circ\Lambda_\xi\circ\su_i).
\]
So $\tilde \Omega(\xi,\eta,(\U_i,\psi_i,\su_i))$ is an arrow from $(\U_i,\psi_i,\su_i)$ to $(\U_i,\psi_i, \Lambda_{\xi\eta}\inv \circ \Lambda_\eta\circ \Lambda_\xi\circ \su_i)$, i.e.
\[
\begin{tikzpicture}
\def \x{3.8}
\def \y{1}
\node (A00) at (0,0)       {$\sH$};
\node (A10) at (\x,0)      {$\U_i\times_\sH\U_i$};
\node (A11) at (\x,\y)     {$\U_i$};
\node (A1-1) at (\x,-1*\y) {$\U_i$};
\node (A30) at (3*\x,0)    {$\sf A$.};
\node at (1.87*\x,0) {$\Downarrow \tilde \Omega(\xi,\eta,(\U_i,\psi_i,\su_i))$};
\path (A00) edge [<-] node [auto] {$\scriptstyle{\psi_i}$} (A11);
\path (A00) edge [<-] node [auto,swap] {$\scriptstyle{\psi_i}$} (A1-1);
\path (A11) edge [->] node [auto] {$\scriptstyle{\su_i}$} (A30);
\path (A1-1) edge [->] node [auto,swap] {$\scriptstyle{\Lambda_{\xi\eta}\inv \circ\Lambda_\eta\circ  \Lambda_\xi\circ\su_i}$} (A30);
\path (A10) edge [->] node [auto] {$\scriptstyle{\pi_i}$} (A11);
\path (A10) edge [->] node [auto,swap] {$\scriptstyle{\pi_i}$} (A1-1);
\end{tikzpicture}
\]
An object in the fiber product $\U_i\times_\sH\U_i$ is of the form
\[
\xymatrix{a_1 \ar@{.>}[rr]^-{\alpha_\sH} && a_2}
\]
with $a_i\in \U^0_i$ and $\alpha_\sH:\psi_i^0(a_1)\rto \psi^0_i(a_2)$ in $\sH$. Since $\psi_i$ is an equivalence, there is an arrow $\alpha_{\U_i}:a_1\rto a_2$ in $\U_i$ satisfying $\psi^1(\alpha_{\U_i})=\alpha_{\sH}$. Then the $\tilde \Omega(\xi,\eta,(\U_i,\psi_i,\su_i))$ we want is a map
\[
\tilde \Omega(\xi,\eta,(\U_i,\psi_i,\su_i)): (\U_i\times_\sH\U_i)^0\rto A^1.
\]
We set
\begin{align*}
\tilde \Omega(\alpha_K,\beta_K,(\U_i,\psi_i,\su_i))(\xymatrix{a_1 \ar@{.>}[r]^-{\alpha_\sH} & a_2})
:=&\, \su_i(\alpha_{\U_i})\cdot \Omega(\xi,\eta,\su_i(a_2))\\
=&\, \Omega(\xi,\eta,\su_i(a_1))\cdot \Lambda_{\xi\eta}\inv\circ \Lambda_\eta\circ\Lambda_\xi\circ \su_i(\alpha_{\U_i}).
\end{align*}
The last equality follows from the fact that $\Omega(\xi,\eta):{\sf id}_\A\Rto \Lambda_{\xi\eta}\inv\circ \Lambda_\eta\circ\Lambda_\xi$ is a natural transformation. By direct computation we have
\begin{theorem}\label{T induced-extension-on-Mor}
The pair $(\tilde \Lambda,\tilde \Omega)$ is a generalized cocycle.
Therefore we get a groupoid fibration $\sf Mor(H,A)\rtimes_{\tilde \Lambda,\tilde \Omega}\K$ over $\K$ with fiber $\sf Mor(H,A)$.

Moreover, when the fundamental class $[\overline\sH]$ exists, $\tilde \Lambda$ preserves $\sf Mor(\sH,\A)_\beta$, therefore we could restrict $(\tilde \Lambda,\tilde \Omega)$ to $\sf Mor(\sH,\A)_\beta$. Denote the restriction by $(\tilde \Lambda_\beta,\tilde\Omega_\beta)$. Then we have a groupoid fibration $\sf Mor(\sH,\A)_\beta\rtimes_{\tilde \Lambda_\beta,\tilde\Omega_\beta}\K$ over $\K$ with fiber $\sf Mor(\sH,\A)_\beta$.
\end{theorem}

\subsection{Fiber class generalized morphisms to $\G$}\label{S subs-fiber-class-generalized-morphisms}

Obviously, every object $(x,(\U,\psi,\su))$ in the groupoid $\sf Mor(H,A)\rtimes_{\tilde \Lambda,\tilde \Omega}\K$ defines a generalized morphism
\[
\xymatrix{\sH & \U\ar[l]_-{\psi}\ar[r]^-\su & \A\ar@{^(->}[r]^-{\tau_x} &\G }
\]
via the inclusion $\tau_x$ of $\A$ as the fiber over $x\in K^0$ (or $(1_x\rrto x)$) of the kernel $\ker\phi$ of the topological product Lie groupoid fibration $\phi:\G\rto\K$. We denoted this induced generalized morphism by $(\U,\psi,\su)_x$. It is obvious that for every object $(x,(\U,\psi,\su))$ in the groupoid $\sf Mor(H,A)\rtimes_{\tilde \Lambda,\tilde \Omega}\K$, the induced generalized morphism $(\U,\psi,\su)_x$ into $\G$ is a fiber class generalized morphism. Moreover this procedure gives rise to a groupoid morphism
\[
\tau:=(\tau^0,\tau^1):{\sf Mor(H,A)}\rtimes_{\tilde \Lambda,\tilde \Omega} \K\rto {\sf Mor(H,G)}_F.
\]
The $\tau^1$ on arrows is given by
\[
\tau^1(\alpha,\xi)(\cdot)=(\alpha(\cdot),\xi):(\U\times_\sH\V)^0\rto G^1=A^1\times K^1,
\]
where $\alpha:(\U\times_\sH\V)^0\rto A^1$ is an arrow between two morphisms $(\U,\psi,\su),(\V,\phi,\sv)\in\mathrm{Mor}^0(\sH,\A)_F$.

Similarly, when the fundamental class $[\overline\sH]$ exists, we have
\[
\tau_\beta: {\sf Mor(H,A)_\beta }\rtimes_{\tilde \Lambda,\tilde \Omega}\K \rto {\sf Mor(H,G)}_\beta.
\]
Our main theorem in this section is
\begin{theorem}\label{T tau-is-equivalence}
Suppose that $\overline\sH$ is compact and connected, and $\K$ is \'etale. Then $\tau$ and $\tau_\beta$ are both equivalence groupoids.
\end{theorem}
\begin{proof}
We prove this theorem for $\tau$. The proof for $\tau_\beta$ is similar. Here the equivalence is the equivalence between categories. So the proof consists of two steps. In the first step we show that $\tau$ is essentially surjective. This is equivalent to that every fiber class generalized morphism $(\U,\psi,\su):\sH\rightharpoonup \G$ is connected to a generalized morphism in $\tau^0(({\sf Mor(H,A)}\rtimes\K)^0)$ by an arrow in $\mathrm{Mor}^1(\sH,\G)_F$. In the second step we show that $\tau$ is full and faithful.

\v
\n{\bf Step I.} Let $(\U,\psi,\su):\sH\rightharpoonup \G$ be a fiber class generalized morphism. Suppose $\psi=(\psi^0,\psi^1),\su=(u^0,u^1)$, $\U=(U^1\rrto U^0)$ and
\begin{align}\label{E decompose-U-0}
U^0=\bigsqcup_{i\in I} U^0_i
\end{align}
be a decomposition of $U^0$ into open, connected components. Since $\overline\sH$ is compact, we could first modify $(\U,\psi,\su)$ into a new generalized morphism $(\V,\varphi,\sv):\sH\rightharpoonup \G$ such that $V^0$ has only finite connected components. This can be done as follows. The decomposition \eqref{E decompose-U-0} of $U^0$ induces an open cover $\{\overline{U^0_i}\}_{i\in I}$ of $\overline\sH$. Since $\overline\sH$ is compact. The open cover $\{\overline{U^0_i}\}_{i\in I}$ has a finite cover, say $\{\overline{U^0_{i_k}}\}_{1\leq k\leq n}$ for an $n\in\integer_{\geq 1}$. Let
\[
V^0:=\bigsqcup_{1\leq k\leq n} U^0_{i_k}
\]
Then we take $\V=\U|_{V^0}$, and $\varphi=\psi|_{\V}, \sv=\su|_\V$. Since $(\U,\psi,\su)$ is a fiber class generalized morphism, so is $(\V,\varphi,\sv)$. On the other hand, the natural inclusion $i^0:V^0\hrto U^0$ induces an equivalence $\sf i:\V\hrto \U$.
Then it gives rise to an arrow in $\mathrm{Mor}^1(\sH,\G)_F$ which connects $(\U,\psi,\su)$ and $(\V,\varphi,\sv)$. So in the following we always assume that for the original $(\U,\psi,\su)$, $U^0$ has finite connected components and $I=\{1,\ldots, n\}$ for some $n\in\integer_{\geq 1}$.

Since $(\U,\psi,\su)$ is a fiber class generalized morphism, $\overline\phi \circ \overline\su (\overline\U )=\{pt\}\in  \overline\K $. Therefore each connected component $U_i^0$ is mapped into a fiber of $G^0\rto K^0$. Suppose $\psi^0(u^0(U^0_i))=x_i\in K^0$.
If all $x_i=x\in K^0$, then $(\U,\psi,\su)\in \text{Im}\,(\tau^0)$. So we next assume that these $x_i$ are not all the same.

Now we modify $\su$ to a new morphism $\w=(w^0,w^1):\U\rto \G$ such that $\phi^0(w^0(U^0))=x_1$. We denote by
$
U^1[U^0_i,U^0_j]
$
the space of arrows in $U^1$ that start from $U^0_i$ and end at $U^0_j$.

For every $x_i, 1\leq i\leq n$, take an arrow $\xi_i:x_i\rto x_1$ in $K^1$.\footnote{Such an arrow exists, since $\overline{\phi}(x_i)=\{pt\}\in\overline{\K}$.} In particular, for every $i$ with $x_i=x_1$ we require that $\xi_i=1_{x_i}=1_{x_1}$. Then we set
\[
w^0(\cdot):=
\Lambda_{\xi_i} \circ u^0(\cdot)=t(1_{u^0(\cdot)},\xi_i) 
\qq \mathrm{on}\qq U^0_i,
\]
and
\[
w^1(\cdot):=(1_{u^0(s(\cdot))},\xi_i)\inv \cdot u^1(\cdot)\cdot (1_{u^0(t(\cdot))},\xi_j)\qq
\mathrm{on}\qq  U^1[U^0_i,U^0_j].
\]
Then one can see that the natural transformation from $\su$ to $\w$ is given by
\[
\bigsqcup_{i\in I} (1_{u^0(\cdot)},\xi_i):\bigsqcup_{i\in I}U^0_i\rto G^1.
\]

\n{\bf Step II.} Note that $\tau$ is injective on objects, hence
we only have to prove that
\[
{\sf Mor(H,A)}_F\rtimes_{\tilde\Lambda,\tilde\Omega}\K \stackrel{\tau}{\cong}
{\sf Mor(H,G)}_F\big|_{\tau ({\mathrm{Mor}^0(\sf H,A)}_F\times K^0)}.
\]
Obviously, $\tau$ is injective on arrows. We next construct an inverse of $\tau$. Take two morphisms $(\U,\psi,\su),(\V,\phi,\sv)\in\mathrm{Mor}^0(\A,\sH)_F$ and two points $x,y\in K^0$. Then we get two morphisms $(\U,\psi,\su)_x,(\V,\phi,\sv)_y\in\mathrm{Mor}^0(\sH,\G)_F$. We denote them by $(\U,\psi,\su_x)$ and $(\V,\phi,\sv_y)$ respectively. Suppose that there is an arrow $\tilde\alpha:(\U,\psi,\su)_x\Rto(\V,\phi,\sv)_y$ in $\mathrm{Mor}(\sH,\G)^1_F$. We next find the pre-image of $\tilde\alpha$ under $\tau$.

First by definition $\tilde\alpha:(\U\times_\sH\V)^0\rto G^1=A^1\times K^1$. From the fact that $u^1_x(a)=(u^1(a),1_x)$ and $v^1_y(b)=(v^1(b),1_y)$ on arrows, we see that $\tilde\alpha$ induce a well defined map
$
\overline{\tilde\alpha}:\overline{\U\times_\sH\V}\rto K^1.
$
Moreover, the image is $K^1(x,y)$. Then since $\overline{\sH}$ is connected and $\K$ is \'etale, we see that $\overline{\tilde\alpha}$ is constant. Suppose its image is $\xi$. Then the image of $\tilde\alpha$ is of the form $(\cdot,\xi)$. So we get a map
\[
\alpha:(\U\times_\sH\V)^0\rto A^0
\qq
\mathrm{by}\qq
\tilde\alpha(\cdot)=(\alpha(\cdot),\xi).
\]
Then from the fact that $\tilde\alpha$ is an arrow between $(\U,\psi,\su)_x$ and $(\V,\phi,\sv)_y$, and $\Omega(1_x,\cdot,a)=\Omega(\cdot, 1_y,a)=1_a$ for every $a\in A^0$, we see that $\alpha$ is an arrow between $(\U,\psi,\su)$ and $(\V,\phi,\sv)$. This gives rise to the inverse of $\tau$ on arrows. So $\tau:{\sf Mor(H,A)}_F\rtimes_{\tilde\Lambda,\tilde\Omega}\K \cong
{\sf Mor(H,G)}_F\big|_{\tau ({\mathrm{Mor}^0(\sf H,A)}_F\times K^0)}$. This finishes the proof.
\end{proof}

By using the axiom of choice we could removed the assumption on the compactness of $\overline\sH$.

\begin{remark}
In \cite{Chen-Du-WangR2019}, the authors also constructed another morphism groupoid $\sf FMor(H,A)$ by using {\em full-morphism} (cf. \cite[Definition 3.8]{Chen-Du-WangR2019}) and {\em strict fiber products} (cf. \cite[Definition 2.10]{Chen-Du-WangR2019}). Moreover, there is a natural equivalence
$
\sf i:\sf FMor(H,A)\rto \sf Mor(H,A)
$ 
(cf. \cite[Theorem 3.15]{Chen-Du-WangR2019}).
Similar as the constructions in Subsection \ref{S generalized-morphism-to-A}, we have
$
{\sf FMor(H,A)}_F$, and $\sf FMor(H,A)_\beta.
$
Moreover, $\sf i$ restricts to groupoid equivalences
\[
{\sf i}_F:{\sf FMor(H,A)}_F\rto {\sf Mor(H,A)}_F,
\qq
\text{and}\qq
{\sf i}_\beta:{\sf FMor(H,A)}_\beta\rto {\sf Mor(H,A)}_\beta.
\]
On the other hand, there are also induced groupoid fibrations over $\K$ with fibers being ${\sf FMor(H,A)}_F$ and ${\sf FMor(H,A)}_\beta$ respectively, i.e. similar results as Theorem \ref{T induced-extension-on-Mor} hold for ${\sf FMor(H,A)}_F$ and $\sf FMor(H,A)_\beta$. So we have
$
{\sf FMor(H,A)}_F\rtimes_{\tilde \Lambda,\tilde \Omega}\K$,
and $\sf FMor(H,A)_\beta\rtimes_{\tilde \Lambda,\tilde \Omega}\K.
$
Moreover, one can see that ${\sf i}_F$ and ${\sf i}_\beta$ give rise to groupoid equivalences
\[
{\sf FMor(H,A)}_F\rtimes_{\tilde \Lambda,\tilde \Omega}\K\rto
{\sf Mor(H,A)}_F\rtimes_{\tilde \Lambda,\tilde \Omega}\K, \qq
\text{and}
\qq \sf FMor(H,A)_\beta\rtimes_{\tilde \Lambda,\tilde \Omega}\K\rto
\sf Mor(H,A)_\beta\rtimes_{\tilde \Lambda,\tilde \Omega}\K.
\]
Finally, one can also see that Theorem \ref{T tau-is-equivalence} holds for both ${\sf FMor(H,A)}_F\rtimes_{\tilde \Lambda,\tilde \Omega}\K$ and $\sf FMor(H,A)_\beta\rtimes_{\tilde \Lambda,\tilde \Omega}\K$. We leave the details to the readers.
\end{remark}

\appendix

\section{Groupoid cohomology and the cocycle $\Xi$}

\subsection{Groupoid cohomology with coefficient in an abelian group}\label{S appendix-groupoid-cohomology}
We fix some notation of groupoid cohomology. Let $\G=(G^1\rrto G^0)$ be a Lie groupoid. Let $E$ be an abelian (Lie) group. A $\G$-action on $E$ means a (smooth) map $\bar\Lambda:G^1\rto\mathrm{Aut}(E)$ satisfying that for every pair of composable arrows $g,h$ in $G^1$, i.e. $t(g)=s(h)$, it holds
\begin{align}\label{E condition-on-groupoid-action}
\bar\Lambda(h)\circ \bar\Lambda(g)=\bar\Lambda(gh).
\end{align}
Recall that the composition of two arrows $gh$ goes from left to right.

Consider the notation $G^{[n]}$ for the set of $n$-composable
arrows\footnote{As we have remark in the definition of multiplication of arrows of groupoids in the beginning of \S \ref{S 2}, the convention is different from the usual one, in which $s(g_i)=t(g_{i+1})$, that is the arrows go from the right to the left. However, we emphasize that it goes from left to right. So in the following, the definition of natural maps from $G^{[n]}$ to $G^0$ will be different from the usual one (see for example \cite{Crainic2003,Bos2013}), so is the definition of coboundary map.}
\begin{align}\label{E appendix-G-[n]}
G^{[n]}:=\{\vec g=(g_1,\ldots,g_n)\mid t(g_i)=s(g_{i+1}), 1\leq i\leq n-1\}.
\end{align}
When $n=0$, $G^{[0]}=G^0$ is the space of objects. We could illustrate a point $\vec g=(g_1,\ldots,g_n)\in G^{[n]}$ by
\[
\xymatrix{x=s(g_1)\ar[r]^-{g_1} & t(g_1)=s(g_2)\ar[r]^-{g_1} & \cdots \ar[r]^{g_n} & t(g_n). }
\]

Now let $E$ be an abelian (Lie) group equipped with a $\G$-action. For $n\in\integer_{\geq 0}$ an {\bf $n$-cochain} on $G$ with values in $E$ is a (smooth) map $G^{[n]}\rto E$. Denote the set of $n$-cochain on $G$ with values in $E$ by $C^n(\G,E)$. The coboundary map
$
d:C^n(\G,E)\rto C^{n+1}(\G,E)
$
that makes the the family $C^*(\G,E)$ into a cochain complex is defined by
\begin{align*}
dc(g_1,\ldots,g_{n+1}):&=\bar\Lambda(g_1)\inv(c(g_2,\ldots,g_{n+1}))
+ \sum_{i=1}^n (-1)^i c(g_1,\ldots,g_ig_{i+1},\ldots, g_{n+1})%
\\
&
+ (-1)^{n+1}c(g_1,\ldots, g_n)
\end{align*}
for $n\geq 1$, with
$
dc(g)=\bar\Lambda(g)\inv (c(t(g)))-c(s(g)).
$
Here we write the operation in $E$ by ``$+$''. Then one can check that $dd=0$ on $C^n(\G,E)$ for all $n\in\integer_{\geq 0}$. The associated {\em groupoid cohomology group} is
\[
H^n_{\bar \Lambda}(\G,E):=\frac{\ker \{d:C^n(\G,E)\rto C^{n+1}(\G,E)\}}{\mathrm{Im}\{d:C^n(\G,E)\rto C^{n+1}(\G,E)\}}
\]
for $n\geq 1$ and $H^0_{\bar \Lambda}(\G,E)=\ker \{d:C^0(\G,E)\rto C^1(\G,E)\}$.

\subsection{Induced action of $\K$ on $Z_\A$ for a topological product Lie groupoid fibration by $\A$}\label{S appdndix-induced-K-action-on-ZA}

Let $\bar \Lambda:K^1\rto \overline{\sf SAut(A)}$ be a morphism with smooth liftings. Then there is an induced $\K$-action on the group $Z_\A$. We next describe this action. We first take a smooth lifting $\Lambda:K^1\rto\mathrm{SAut}^0(\A)$. So it assigns an arrow $\xi:x\rto y$ in $K^1$ an isomorphism $\Lambda_\xi:\A\rto\A$. For every $\sigma\in Z_\A$, we define $\Lambda_\xi(\sigma)$ by
\begin{align}\label{E def-S-action-on-KA}
\Lambda_\xi(\sigma)(a):=\Lambda_\xi\circ \sigma\circ \Lambda_\xi\inv(a)\stackrel{\text{abbreviation}}{=}\Lambda_\xi\sigma \Lambda_\xi\inv(a).
\end{align}

\begin{lemma}
The formula \eqref{E def-S-action-on-KA} defines a $\K$-action on $Z_\A$.
\end{lemma}
\begin{proof}
As in \eqref{E def-Omega-2} for every pair of composable arrows $\xi,\eta$ in $K^1$, we have the natural transformation $\Omega(\xi,\eta):{\sf id_A}\Rto \Lambda_{\xi\eta}\inv\Lambda_\eta \Lambda_\xi$. We will show that the $\A$-invariance of $\sigma\in Z_\A$ ensure that
$
\Lambda_\eta(\Lambda_\xi(\sigma))=\Lambda_{\xi\eta}(\sigma).
$

Take an element $a\in A^0$. Then
\begin{align*}
&\Lambda_\eta(\Lambda_\xi(\sigma))(a)
=\Lambda_\eta \Lambda_\xi(\sigma)(\Lambda_\eta\inv (a))
=\Lambda_\eta \Lambda_\xi \sigma \Lambda_\xi\inv \Lambda_\eta\inv(a),\qq
\mathrm{and}\qq \Lambda_{\xi\eta}(\sigma)(a)=
\Lambda_{\xi\eta}\sigma\Lambda_{\xi\eta}\inv (a).
\end{align*}
Since $\Omega(\xi,\eta):{\sf id_A}\Rto \Lambda_{\xi\eta}\inv \Lambda_\eta\Lambda_\xi$ is a natural transformation we have a commutative diagram
\[
\xymatrix{
\Lambda_\xi\inv \Lambda_\eta\inv(a) \ar[rrr]^-{\Omega(\xi,\eta,\Lambda_\xi\inv\Lambda_\eta\inv(a))}
\ar[d]_-{\sigma(\Lambda_\xi\inv\Lambda_\eta\inv(a))} & &&
\Lambda_{\xi\eta}\inv(a)
\ar[d]^-{\Lambda_{\xi\eta}\inv \Lambda_\eta\Lambda_\xi(\sigma(\Lambda_\xi\inv \Lambda_\eta\inv(a)))}\\
\Lambda_\xi\inv\Lambda_\eta\inv(a) \ar[rrr]^-{\Omega(\xi,\eta,\Lambda_\xi\inv\Lambda_\eta\inv(a))} && &
\Lambda_{\xi\eta}\inv(a).}
\]
The $\A$-invariance of $\sigma$ implies that for every arrow $h:\Lambda_\xi\inv \Lambda_\eta\inv(a) \rto \Lambda_{\xi\eta}\inv(a)$ in $\A$, we have
\[
h\inv \cdot \sigma (\Lambda_\xi\inv \Lambda_\eta\inv(a))\cdot h=\sigma(\Lambda_{\xi\eta}\inv(a)).
\]
Therefore for the arrow $\Omega(\xi,\eta,\Lambda_\xi\inv \Lambda_\eta\inv(a)):\Lambda_\xi\inv\Lambda_\eta\inv(a) \rto \Lambda_{\xi\eta}\inv(a)$, the commutative diagram above implies
\begin{align*}
\Lambda_{\xi\eta}\inv \Lambda_\eta\Lambda_\xi(\sigma \Lambda_{\xi}\inv \Lambda_\eta\inv(a))
&=\Omega(\xi,\eta,\Lambda_\xi\inv \Lambda_\eta\inv(a))\inv \cdot\sigma(\Lambda_\xi\inv\Lambda_\eta\inv(a))\cdot \Omega(\xi,\eta,\Lambda_\xi\inv\Lambda_\eta\inv(a))
=\sigma(\Lambda_{\xi\eta}\inv(a)).
\end{align*}
So we have
$
\Lambda_\eta \Lambda_\xi\sigma \Lambda_\xi\inv \Lambda_\eta\inv(a)=\Lambda_{\xi\eta}\sigma \Lambda_{\xi\eta}\inv (a).
$
Consequently $\Lambda_\eta(\Lambda_\xi(\sigma))=\Lambda_{\xi\eta}(\sigma)$. This shows that $\Lambda$ induces a $\K$-action on $Z_\A$.
\end{proof}

Moreover,
\begin{lemma}
The action \eqref{E def-S-action-on-KA} does not depend on the smooth lifting $\Lambda$ of $\bar \Lambda$, but only on $\bar \Lambda$ itself.
\end{lemma}
\begin{proof}
Suppose both $\Lambda_1$ and $\Lambda_2$ are liftings of $\bar \Lambda$. Then there are natural transformations $\rho_\xi:\Lambda_{1,\xi}\Rto \Lambda_{2,\xi}$ for all $\xi\in K^1$. Take an $a\in A^0$. Let $a_i=\Lambda_i\inv(a)$ for $i=1,2$. By the definition \eqref{E def-S-action-on-KA}, we have
\[
\Lambda_{i,\xi}(\sigma)(a)=\Lambda_{i,\xi}\sigma(a_i).
\]
We next show that $\Lambda_{1,\xi}\sigma(a_1))=\Lambda_{2,\xi}\sigma(a_2)$.

First of all, by the natural transformation $\rho_\xi:\Lambda_{1,\xi}\Rto\Lambda_{2,\xi}$ we get an arrow
\[
\rho_\xi(a_2):\Lambda_{1,\xi}(a_2)\rto\Lambda_{2,\xi}(a_2) =\Lambda_{1,\xi}(a_1)=a,
\]
which induces an arrow $\Lambda_{1,\xi}\inv[\rho_\xi(a_2)]:a_2\rto a_1$. Therefore by the $\A$-invariance of $\sigma$ we get
\[
\sigma(a_1)=\Lambda_{1,\xi}\inv[\rho_\xi(a_2)]\inv\cdot \sigma(a_2) \cdot \Lambda_{1,\xi}\inv[\rho_\xi(a_2)].
\]
Applying $\Lambda_{1,\xi}$ to both sides we get
\[
\Lambda_{1,\xi}(\sigma(a_1))=\rho_\xi(a_2)\inv \cdot \Lambda_{1,\xi}(\sigma(a_2))\cdot \rho_\xi(a_2).
\]

Again, since $\rho_\xi :\Lambda_{1,\xi}\Rto \Lambda_{2,\xi}$, for the arrow $\sigma(a_2):a_2\rto a_2$ we have a commutative diagram in $A^1$
\[
\xymatrix{
\Lambda_{1,\xi}(a_2) \ar[d]_-{\Lambda_{1,\xi}(\sigma(a_2))} \ar[rr]^-{\rho_\xi(a_2)} &&
\Lambda_{2,\xi}(a_2) \ar[d]^-{\Lambda_{2,\xi}(\sigma(a_2))}\\
\Lambda_{1,\xi}(a_2)  \ar[rr]^-{\rho_\xi(a_2)} &&
\Lambda_{2,\xi}(a_2) .}
\]
Therefore
\[
\Lambda_{1,\xi}(\sigma(a_1))=\rho_\xi(a_2)\inv \cdot \Lambda_{1,\xi}(\sigma(a_2))\cdot \rho_\xi(a_2)\inv=\Lambda_{2,\xi}(\sigma(a_2)).
\]
This finishes the proof.
\end{proof}

So we denote this induced action of $\K$ on $Z_\A$ by $
\bar \Lambda:K^1\rto\mathrm{Aut}(Z_\A)$. By using this action we could define groupoid cohomology $H^*_{\bar \Lambda}(\K,Z_\A)$ of $\K$ with coefficients in $Z_\A$.

\subsection{The cocycle $\Xi$}\label{S appendix-Xi}
Recall that given a smooth lifting $\Lambda:K^1\rto\mathrm{SAut}^0(\A)$ of $\bar \Lambda$, and the corresponding smooth family of natural transformation $\Omega:K^{[2]}\times A^0\rto A^1$, we have the following elements of $Z_\A$ (cf. \eqref{E def-Xi})
\begin{align*}
\Xi(\xi,\eta,\zeta,a)
=&\, \Omega(\xi,\eta,a)\cdot \Omega(\xi\eta,\zeta,\Lambda_{\xi\eta}\inv \Lambda_\eta\Lambda_\xi(a))
\cdot \Omega(\xi,\eta\zeta,\Lambda_\xi\inv\Lambda_{\eta\zeta}\inv \Lambda_\zeta\Lambda_\eta\Lambda_\xi(a))\inv
\cdot \Lambda_\xi\inv[\Omega(\eta,\zeta,\Lambda_\xi(a))\inv]
\end{align*}
for three composable arrows $\xi,\eta,\zeta$ of $\K$. Therefore $\Xi$ is a cochain in $C^3_{\bar \Lambda}(\K,Z_\A)$.

\begin{theorem}\label{T Apendix-Xi}
$\Xi$ is a 3-cocycle in the cochain complex $C^*_{\bar \Lambda}(\K,Z_\A)$.
\end{theorem}
Before we prove this theorem, we first analyze the properties of $\Xi$.

\begin{lemma}
We could cyclically permute the elements in the product of the definition of $\Xi$ without change the values. Take three composable arrows $\xi,\eta,\zeta$ in $\K$. For any $a\in A^0$, set
\begin{eqnarray*}
\Xi_1(\xi,\eta,\zeta,a)
&=&\Omega(\xi\eta,\zeta,a)
\cdot \Omega(\xi,\eta\zeta,\Lambda_\xi\inv \Lambda_{\eta\zeta}\inv \Lambda_\zeta \Lambda_{\xi\eta}(a))\inv
\\&&
\cdot \Lambda_\xi\inv[\Omega(\eta,\zeta,\Lambda_\eta\inv \Lambda_{\xi\eta}(a))\inv]
\cdot \Omega(\xi,\eta,\Lambda_\xi\inv\Lambda_\eta\inv \Lambda_{\xi\eta}(a)),\\
\Xi_2(\xi,\eta,\zeta,a)
&=&\Omega(\xi,\eta\zeta,\Lambda_\xi\inv \Lambda_{\eta\zeta}\inv \Lambda_{\xi\eta\zeta}(a))\inv
\cdot \Lambda_\xi\inv[\Omega(\eta,\zeta,\Lambda_\eta\inv \Lambda_\zeta\inv \Lambda_{\xi\eta\zeta}(a))\inv]
\\&&
\cdot \Omega(\xi,\eta,\Lambda_\xi\inv \Lambda_\eta\inv \Lambda_\zeta\inv\Lambda_{\xi\eta\zeta}(a))
\cdot\Omega(\xi\eta,\zeta,\Lambda_{\xi\eta}\inv \Lambda_\zeta\inv \Lambda_{\xi\eta\zeta}(a)),
\\
\Xi_3(\xi,\eta,\zeta,a)
&=&\Lambda_\xi\inv[\Omega(\eta,\zeta,\Lambda_\eta\inv \Lambda_\zeta\inv \Lambda_{\eta\zeta}\Lambda_\xi(a))\inv]
\cdot \Omega(\xi,\eta,\Lambda_\xi\inv\Lambda_\eta\inv \Lambda_\zeta\inv\Lambda_{\eta\zeta}\Lambda_\xi(a))
\\&&
\cdot\Omega(\xi\eta,\zeta,\Lambda_{\xi\eta}\inv \Lambda_\zeta\inv \Lambda_{\eta\zeta}\Lambda_\xi(a))
\cdot \Omega(\xi,\eta\zeta,a)\inv
\end{eqnarray*}
Then $\Xi=\Xi_1=\Xi_2=\Xi_3$.
\end{lemma}
\begin{proof}
We show that $\Xi_1$, $\Xi_2$ and $\Xi_3$ give rise to maps from $K^{[3]}$ to $Z_\A$. Then by the $\A$-invariance of elements of $Z_\A$. We get this lemma.

We give the explicit computation about $\Xi_1$. We first show that it belongs to the center. Take an arrow $\alpha\in \Gamma_a$. By applying \eqref{E 3.3} repeatedly, we have
\begin{eqnarray*}
&&
\alpha\cdot \Omega(\xi\eta,\zeta,a)
\cdot \Omega(\xi,\eta\zeta,\Lambda_\xi\inv\Lambda_{\eta\zeta}\inv \Lambda_\zeta\Lambda_{\xi\eta}(a))\inv
\cdot \Lambda_\xi\inv[\Omega(\eta,\zeta,\Lambda_\eta\inv \Lambda_{\xi\eta}(a))\inv]\cdot \Omega(\xi,\eta,\Lambda_\xi\inv \Lambda_\eta\inv\Lambda_{\xi\eta}(a))\\
&\stackrel{\eqref{E 3.3}}{=}&
\left\{\Omega(\xi\eta,\zeta,a)
\cdot
\Lambda_{\xi\eta\zeta}\inv \Lambda_\zeta\Lambda_{\xi\eta}(\alpha)
\right\}
\cdot \Omega(\xi,\eta\zeta,\Lambda_\xi\inv \Lambda_{\eta\zeta}\inv \Lambda_\zeta \Lambda_{\xi\eta}(a))\inv
\\&&
\cdot \Lambda_\xi\inv[\Omega(\eta,\zeta,\Lambda_\eta\inv \Lambda_{\xi\eta}(a))\inv]\cdot \Omega(\xi,\eta,\Lambda_\xi\inv \Lambda_\eta\inv \Lambda_{\xi\eta}(a))
\\
&\stackrel{\eqref{E 3.3}}{=}&
\Omega(\xi\eta,\zeta,a)\cdot \left\{\Omega(\xi,\eta\zeta,\Lambda_\xi\inv\Lambda_{\eta\zeta}\inv \Lambda_\zeta \Lambda_{\xi\eta}(a))\inv
\cdot
(\Lambda_{\xi\eta\zeta}\inv \Lambda_{\eta\zeta}\Lambda_\xi)\inv \Lambda_{\xi\eta\zeta}\inv \Lambda_\zeta\Lambda_{\xi\eta}(\alpha)
\right\}
\\&&
\cdot \Lambda_\xi\inv[\Omega(\eta,\zeta,\Lambda_\eta\inv \Lambda_{\xi\eta}(a))\inv]\cdot \Omega(\xi,\eta,\Lambda_\xi\inv \Lambda_\eta\inv\Lambda_{\xi\eta}(a))
\\
&=&
\Omega(\xi\eta,\zeta,a)\cdot \Omega(\xi,\eta\zeta,\Lambda_\xi\inv\Lambda_{\eta\zeta}\inv \Lambda_\zeta\Lambda_{\xi\eta}(a))\inv
\\&&
\cdot
\left\{\Lambda_\xi\inv \Lambda_{\eta\zeta}\inv \Lambda_\zeta\Lambda_{\xi\eta}(\alpha)\right\}
\cdot \Lambda_\xi\inv\left[\Omega(\eta,\zeta,\Lambda_\eta\inv \Lambda_{\xi\eta}(a))\inv\right]
\cdot \Omega(\xi,\eta,\Lambda_\xi\inv\Lambda_\eta\inv
\Lambda_{\xi\eta}(a))\\
&=&\Omega(\xi\eta,\zeta,a)\cdot \Omega(\xi,\eta\zeta,\Lambda_\xi\inv\Lambda_{\eta\zeta}\inv \Lambda_\zeta\Lambda_{\xi\eta}(a))\inv
\\&&
\cdot
\left\{\Lambda_\xi\inv \left[
\Lambda_{\eta\zeta}\inv \Lambda_\zeta\Lambda_{\xi\eta}(\alpha)\cdot \Omega(\eta,\zeta,\Lambda_\eta\inv \Lambda_{\xi\eta}(a))\inv
\right]\right\}
\cdot \Omega(\xi,\eta,\Lambda_\xi\inv\Lambda_\eta\inv \Lambda_{\xi\eta}(a))
\\
&\stackrel{\eqref{E 3.3}}{=}&\Omega(\xi\eta,\zeta,a)\cdot \Omega(\xi,\eta\zeta,\Lambda_\xi\inv\Lambda_{\eta\zeta}\inv \Lambda_\zeta\Lambda_{\xi\eta}(a))\inv 
\cdot
\\&&
\Lambda_\xi\inv\left\{ \Omega(\eta,\zeta,\Lambda_\eta\inv \Lambda_{\xi\eta}(a))\inv \cdot (\Lambda_{\eta\zeta}\inv \Lambda_\zeta\Lambda_\eta)\inv \Lambda_{\eta\zeta}\inv \Lambda_\zeta\Lambda_{\xi\eta}(\alpha)\right\}
\cdot \Omega(\xi,\eta,\Lambda_\xi\inv\Lambda_\eta\inv \Lambda_{\xi\eta}(a))
\\
&=&\Omega(\xi\eta,\zeta,a)\cdot \Omega(\xi,\eta\zeta,\Lambda_\xi\inv \Lambda_{\eta\zeta}\inv \Lambda_\zeta \Lambda_{\xi\eta}(a))\inv %
\\&&
\cdot
\left\{\Lambda_\xi\inv [\Omega(\eta,\zeta,\Lambda_\eta\inv \Lambda_{\xi\eta}(a))\inv]
\cdot
\Lambda_\xi\inv \Lambda_\eta\inv \Lambda_{\xi\eta}(\alpha)
\right\}
\cdot \Omega(\xi,\eta,\Lambda_\xi\inv \Lambda_\eta\inv \Lambda_{\xi\eta}(a))
\\
&\stackrel{\eqref{E 3.3}}{=}&\Omega(\xi\eta,\zeta,a)\cdot \Omega(\xi,\eta\zeta,\Lambda_\xi\inv\Lambda_{\eta\zeta}\inv \Lambda_\zeta\Lambda_{\xi\eta}(a))\inv \cdot
\Lambda_\xi\inv [\Omega(\eta,\zeta,\Lambda_\eta\inv \Lambda_{\xi\eta}(a))\inv]
\\&&
\cdot \left\{\Omega(\xi,\eta,\Lambda_\xi\inv\Lambda_\eta\inv \Lambda_{\xi\eta}(a))
\cdot
\Lambda_{\xi\eta}\inv \Lambda_\eta\Lambda_\xi \Lambda_\xi\inv \Lambda_\eta\inv \Lambda_{\xi\eta}(\alpha)
\right\}
\\
&=&\Omega(\xi\eta,\zeta,a)\cdot \Omega(\xi,\eta\zeta,\Lambda_\xi\inv\Lambda_{\eta\zeta}\inv \Lambda_\zeta\Lambda_{\xi\eta}(a))\inv
\\&&
\cdot
\Lambda_\xi\inv [\Omega(\eta,\zeta,\Lambda_\eta\inv \Lambda_{\xi\eta}(a))\inv]  \cdot \Omega(\xi,\eta,\Lambda_\xi\inv \Lambda_\eta\inv \Lambda_{\xi\eta}(a))\cdot \{\alpha\}.
\end{eqnarray*}
Therefore, $\Xi_1(\xi,\eta,\zeta,a)\in ZA^0$. One can also see that, it is $\A$-invariant, hence gives rise to a map $K^{[3]}\rto Z_A$. From the expression of $\Xi$ and $\Xi_1$ we have
\begin{align*}
\Xi(\xi,\eta,\zeta,a)
&= \Omega(\xi,\eta,a)\cdot
\Xi_1(\xi,\eta,\zeta,\Lambda_{\xi\eta}\inv \Lambda_\eta\Lambda_\xi(a))
\cdot \Omega(\xi,\eta,\Lambda_\xi\inv\Lambda_\eta\inv \Lambda_{\xi\eta}(\Lambda_{\xi\eta}\inv \Lambda_\eta\Lambda_\xi(a)))\inv \\
&=\Omega(\xi,\eta,a)\cdot
\Xi_1(\xi,\eta,\zeta,\Lambda_{\xi\eta}\inv \Lambda_\eta\Lambda_\xi(a))
\cdot \Omega(\xi,\eta, a)\inv .
\end{align*}
Then since $\Xi_1$ has images in $Z_\A$, we also have
\begin{align*}
\Xi_1(\xi,\eta,\zeta,a)
&=\Omega(\xi,\eta,a)\cdot
\Xi_1(\xi,\eta,\zeta,\Lambda_{\xi\eta}\inv \Lambda_\eta\Lambda_\xi(a))
\cdot \Omega(\xi,\eta, a)\inv
\end{align*}
Consequently,
$
\Xi(\xi,\eta,\zeta,a)=\Xi_1(\xi,\eta,\zeta,a).
$
By similar computations, we see that $\Xi_2$ and $\Xi_3$ also belong to $Z_\A$, and $\Xi_2=\Xi_3=\Xi$. This finishes the proof.
\end{proof}

Now we proceed to prove Theorem \ref{T Apendix-Xi}. We compute the differential of $\Xi$. Since the value of $\Xi$ lies in the center, we have
\begin{align*}
d\,\Xi(\xi_1,\xi_2,\xi_3,\xi_4)
=& \Lambda_{\xi_1}\inv(\Xi(\xi_2,\xi_3,\xi_4))\cdot \Xi(\xi_1\xi_2,\xi_3,\xi_4)\inv
\cdot\Xi(\xi_1,\xi_2\xi_3,\xi_4)\cdot \Xi(\xi_1,\xi_2,\xi_3\xi_4)\inv\cdot \Xi(\xi_1,\xi_2,\xi_3)\\
=&\Lambda_{\xi_1}\inv\Xi(\xi_2,\xi_3,\xi_4)\cdot
\Xi(\xi_1\xi_2,\xi_3,\xi_4)\inv\cdot
\Xi(\xi_1,\xi_2,\xi_3\xi_4)\inv
\cdot \Xi(\xi_1,\xi_2\xi_3,\xi_4) \cdot \Xi(\xi_1,\xi_2,\xi_3).
\end{align*}
In the following we use some abbreviations of notations. For example we write
\[
\Lambda_i:=\Lambda_{\xi_i},\qq \xi_{12}:=\xi_1\xi_2.
\]
We first compute $ \Xi(\xi_1\xi_2,\xi_3,\xi_4)\inv\cdot \Xi(\xi_1,\xi_2,\xi_3\xi_4)\inv\cdot \Xi(\xi_1,\xi_2\xi_3,\xi_4)$. Take an object $a\in A^0$, we have
\begin{eqnarray*}
\Upsilon&:=&\Xi(\xi_1\xi_2,\xi_3,\xi_4)(a)\inv\cdot \Xi(\xi_1,\xi_2,\xi_3\xi_4)(a)\inv\cdot \Xi(\xi_1,\xi_2\xi_3,\xi_4)(a)\\
&=&\Xi_2(\xi_1\xi_2,\xi_3,\xi_4)(a)\inv\cdot \Xi_2(\xi_1,\xi_2,\xi_3\xi_4)(a)\inv\cdot \Xi_2(\xi_1,\xi_2\xi_3,\xi_4)(a)\\
&=&\Big\{\Omega(\xi_{123},\xi_4, \Lambda_{123}\inv \Lambda_4\inv \Lambda_{1234}(a))\inv \cdot \Omega(\xi_{12},\beta_3,\Lambda_{12}\inv \Lambda_3\inv \Lambda_4\inv \Lambda_{1234}(a))\inv
\\&&
\cdot
\Lambda_{12}\inv[\Omega(\xi_3,\xi_4, \Lambda_3\inv \Lambda_4\inv \Lambda_{1234}(a))] \cdot \Omega(\xi_{12},\xi_{34},\Lambda_{12}\inv \Lambda_{34}\inv \Lambda_{1234}(a)) \Big\}
\\&&
\cdot
\Big\{\Omega(\xi_{12},\xi_{34}, \Lambda_{12}\inv \Lambda_{34}\inv \Lambda_{1234}(a))\inv \cdot \Omega(\xi_1,\xi_2,\Lambda_1\inv \Lambda_2\inv \Lambda_{34}\inv \Lambda_{1234}(a))\inv
\\&&
\cdot
\Lambda_1\inv[\Omega(\xi_2,\xi_{34}, \Lambda_2\inv \Lambda_{34}\inv \Lambda_{1234}(a))]\cdot \Omega(\xi_1,\xi_{234},\Lambda_1\inv \Lambda_{234}\inv \Lambda_{1234}(a)) \Big\}
\\
&&
\cdot
\Big\{\Omega(\xi_1,\xi_{234},\Lambda_1\inv \Lambda_{234}\inv \Lambda_{1234}(a))\inv  \cdot \Lambda_1\inv[\Omega(\xi_{23},\xi_4, \Lambda_{23}\inv \Lambda_4\inv \Lambda_{1234}(a))\inv ]
\\&&
\cdot
\Omega(\xi_1,\xi_{23},\Lambda_1\inv \Lambda_{23}\inv \Lambda_4\inv  \Lambda_{1234}(a))\cdot\Omega(\xi_{123},\xi_4, \Lambda_{123}\inv \Lambda_4\inv S_{1234}(a))\Big\}
\\
&=
&
\Omega(\xi_{123},\xi_4, \Lambda_{123}\inv \Lambda_4\inv \Lambda_{1234}(a))\inv
\cdot
\Big\{\Omega(\xi_{12},\xi_3,\Lambda_{12}\inv \Lambda_3\inv \Lambda_4\inv \Lambda_{1234}(a))\inv
\\&&
\cdot \Lambda_{12}\inv[\Omega(\xi_3,\xi_4, \Lambda_3\inv \Lambda_4\inv \Lambda_{1234}(a))]
\cdot
\Omega(\xi_1,\xi_2,\Lambda_1\inv \Lambda_2\inv \Lambda_{34}\inv \Lambda_{1234}(a))\inv
\\&&
\cdot \Lambda_1\inv [\Omega(\xi_2,\xi_{34},\Lambda_2\inv \Lambda_{34}\inv \Lambda_{1234}(a))]
\cdot
\Lambda_1\inv[\Omega(\xi_{23},\xi_4, \Lambda_{23}\inv \Lambda_4\inv  \Lambda_{1234}(a))\inv ]
\\&&
\cdot
\Omega(\xi_1,\xi_{23}, \Lambda_1\inv \Lambda_{23}\inv \Lambda_4\inv \Lambda_{1234}(a))\Big\}
\cdot
\Omega(\xi_{123},\xi_4,\Lambda_{123}\inv \Lambda_4\inv \Lambda_{1234}(a))
\end{eqnarray*}
Denote the middle term by
\begin{align*}
\mc A(\xi_1,\xi_2,\xi_3,\xi_4)(\Lambda_{123}\inv \Lambda_4\inv \Lambda_{1234}(a))
:=\,&\Omega(\xi_{12},\xi_3,\Lambda_{12}\inv \Lambda_3\inv \Lambda_4\inv \Lambda_{1234}(a))\inv \cdot \Lambda_{12}\inv[\Omega(\xi_3,\xi_4, \Lambda_3\inv \Lambda_4\inv \Lambda_{1234}(a))]
\\& \cdot
\Omega(\xi_1,\xi_2,\Lambda_1\inv \Lambda_2\inv \Lambda_{34}\inv \Lambda_{1234}(a))\inv \cdot \Lambda_1\inv [\Omega(\xi_2,\xi_{34},\Lambda_2\inv \Lambda_{34}\inv \Lambda_{1234}(a))]
\\& \cdot
\Lambda_1\inv[\Omega(\xi_{23},\xi_4, \Lambda_{23}\inv \Lambda_4\inv  \Lambda_{1234}(a))\inv ]\cdot
\Omega(\xi_1,\xi_{23}, \Lambda_1\inv \Lambda_{23}\inv \Lambda_4\inv \Lambda_{1234}(a)).
\end{align*}
Then one finds that $\mc A(\xi_1,\xi_2,\xi_3,\xi_4)(\Lambda_{123}\inv \Lambda_4\inv \Lambda_{1234}(a))$ belongs to the center of $\Gamma_{\Lambda_{123}\inv \Lambda_4\inv \Lambda_{1234}(a)}$. And by \eqref{E 3.3}, one finds that $\mc A(\xi_1,\xi_2,\xi_3,\xi_4)$ gives rise to an $\A$-invariant section of $\Z\A\rto\A$, hence belongs to $Z_\A$. Therefore $\Upsilon$ is equal to
\begin{align*}
\mc A(\xi_1,\xi_2,\xi_3,\xi_4)(a)
=\,&\Omega(\xi_{12},\xi_3,\Lambda_{12}\inv \Lambda_3\inv \Lambda_{123}(a))\inv \cdot \Lambda_{12}\inv(\Omega(\xi_3,\xi_4, \Lambda_3\inv \Lambda_{123}(a)))
\\&
\cdot
\Omega(\xi_1,\xi_2,\Lambda_1\inv \Lambda_2\inv \Lambda_{34}\inv \Lambda_4 \Lambda_{123}(a))\inv \cdot \Lambda_1\inv (\Omega(\xi_2,\xi_{34},\Lambda_2\inv \Lambda_{34}\inv \Lambda_4 \Lambda_{123}(a)))
\\&
\cdot
\Lambda_1\inv(\Omega(\xi_{23},\xi_4, \Lambda_{23}\inv \Lambda_{123}(a))\inv )\cdot
\Omega(\xi_1,\xi_{23}, \Lambda_1\inv \Lambda_{23}\inv \Lambda_{123}(a)),
\end{align*}
where we replace the $a$ in the expression of $\mc A(\xi_1,\xi_2,\xi_3,\xi_4)(\Lambda_{123}\inv \Lambda_4\inv \Lambda_{1234}(a))$ by $\Lambda_{1234}\inv \Lambda_4 \Lambda_{123}(a)$.
We next multiply this term with $\Xi(\xi_1,\xi_2,\xi_3)(a)$ and get
\begin{eqnarray*}
\Theta&:=& \Upsilon\cdot \Xi(\xi_1,\xi_2,\xi_3)(a)\\
&=&\Omega(\xi_{12},\xi_3,\Lambda_{12}\inv \Lambda_3\inv \Lambda_{123}(a))\inv \cdot \Lambda_{12}\inv(\Omega(\xi_3,\xi_4, \Lambda_3\inv \Lambda_{123}(a)))
\\&&
\cdot
\Omega(\xi_1,\xi_2,\Lambda_1\inv \Lambda_2\inv \Lambda_{34}\inv \Lambda_4 \Lambda_{123}(a))\inv \cdot \Lambda_1\inv [\Omega(\xi_2,\xi_{34},\Lambda_2\inv \Lambda_{34}\inv \Lambda_4 \Lambda_{123}(a))]
\\&& \cdot
\Lambda_1\inv[\Omega(\xi_{23},\xi_4, \Lambda_{23}\inv \Lambda_{123}(a))\inv ] \cdot
\Omega(\xi_1,\xi_{23}, \Lambda_1\inv \Lambda_{23}\inv \Lambda_{123}(a))
\\&&
\cdot
\Omega(\xi_1,\xi_{23},\Lambda_1\inv \Lambda_{23}\inv \Lambda_{123}(a))\inv \cdot \Lambda_1\inv [\Omega(\xi_2,\xi_3,\Lambda_2\inv \Lambda_3\inv \Lambda_{123}(a))\inv ]\\
&& \cdot
\Omega(\xi_1,\xi_2,\Lambda_1\inv \Lambda_2\inv \Lambda_3\inv \Lambda_{123}(a))\cdot \Omega(\xi_{12},\xi_3,\Lambda_{12}\inv \Lambda_3\inv \Lambda_{123}(a))\\
&=
&\Omega(\xi_{12},\xi_3,\Lambda_{12}\inv \Lambda_3\inv \Lambda_{123}(a))\inv \cdot
\Big\{\Lambda_{12}\inv[\Omega(\xi_3,\xi_4, \Lambda_3\inv \Lambda_{123}(a))]
\\&& \cdot
\Omega(\xi_1,\xi_2,\Lambda_1\inv \Lambda_2\inv \Lambda_{34}\inv \Lambda_4 \Lambda_{123}(a))\inv \cdot \Lambda_1\inv [\Omega(\xi_2,\xi_{34},\Lambda_2\inv \Lambda_{34}\inv \Lambda_4 \Lambda_{123}(a))]
\\&& \cdot
\Lambda_1\inv[\Omega(\xi_{23},\xi_4, \Lambda_{23}\inv \Lambda_{123}(a))\inv ]\cdot \Lambda_1\inv [\Omega(\xi_2,\xi_3,\Lambda_2\inv \Lambda_3\inv \Lambda_{123}(a))\inv ]
\\
&& \cdot
\Omega(\xi_1,\xi_2,\Lambda_1\inv \Lambda_2\inv \Lambda_3\inv \Lambda_{123}(a))\Big\}
\cdot \Omega(\xi_{12},\xi_3,\Lambda_{12}\inv \Lambda_3\inv \Lambda_{123}(a)).
\end{eqnarray*}
As above, denote the middle term by $\mc B(\xi_1,\xi_2,\xi_3,\xi_4)(\Lambda_{12}\inv \Lambda_3\inv \Lambda_{123}(a))$. Then one can also see that it gives rise to an $\A$-invariant section $\mc B(\xi_1,\xi_2,\xi_3,\xi_4)\in Z_\A$. So
\begin{eqnarray*}
\Theta&=&\mc B(\xi_1,\xi_2,\xi_3,\xi_4)(a)\\
&=&
\Lambda_{12}\inv[\Omega(\xi_3,\xi_4, \Lambda_{12}(a))]
\cdot
\Omega(\xi_1,\xi_2,\Lambda_1\inv \Lambda_2\inv \Lambda_{34}\inv \Lambda_4 \Lambda_3 \Lambda_{12}(a))\inv
\\&& \cdot
\Lambda_1\inv [\Omega(\xi_2,\xi_{34},\Lambda_2\inv \Lambda_{34}\inv \Lambda_4 \Lambda_3 \Lambda_{12}(a))]
\cdot
\Lambda_1\inv[\Omega(\xi_{23},\xi_4,\Lambda_{23}\inv \Lambda_3 \Lambda_{12}(a))\inv]
\\&& \cdot
\Lambda_1\inv[\Omega(\xi_2,\xi_3,\Lambda_2\inv \Lambda_{12}(a))\inv]
\cdot \Omega(\xi_1,\xi_2,\Lambda_1\inv \Lambda_2\inv \Lambda_{12}(a)).
\end{eqnarray*}
We rewrite the RHS of this equality as
\begin{eqnarray*}
\Theta&=&\Lambda_{12}\inv[\Omega(\xi_3,\xi_4, \Lambda_{12}(a))]
\cdot
\Big\{\Omega(\xi_1,\xi_2,\Lambda_1\inv \Lambda_2\inv \Lambda_{34}\inv \Lambda_4 \Lambda_3 \Lambda_{12}(a))\inv
\\&& \cdot
\Lambda_1\inv [\Omega(\xi_2,\xi_{34},\Lambda_2\inv \Lambda_{34}\inv \Lambda_4 \Lambda_3 \Lambda_{12}(a))]
\cdot
\Lambda_1\inv[\Omega(\xi_{23},\xi_4,\Lambda_{23}\inv \Lambda_3 \Lambda_{12}(a))\inv]
\\&& \cdot
\Lambda_1\inv[\Omega(\xi_2,\xi_3,\Lambda_2\inv \Lambda_{12}(a))\inv]
\cdot \Omega(\xi_1,\xi_2,\Lambda_1\inv \Lambda_2\inv \Lambda_{12}(a))
\\
&& \cdot \Lambda_{12}\inv[\Omega(\xi_3,\xi_4, \Lambda_{12}(a))]\Big\}
\cdot \Lambda_{12}\inv[\Omega(\xi_3,\xi_4, \Lambda_{12}(a))]\inv.
\end{eqnarray*}
As above, denote the middle term by $\mc C(\xi_1,\xi_2,\xi_3,\xi_4)(\Lambda_{12}\inv  \Lambda_{34}\inv \Lambda_4 \Lambda_3 \Lambda_{12}(a))$.
Then one also see that it gives rise to an $\A$-invariant section in $Z_\A$. Therefore
\begin{eqnarray*}
\Theta&=&\mc C(\xi_1,\xi_2,\xi_3,\xi_4)(a)\\
&=&\Omega(\xi_1,\xi_2,\Lambda_1\inv \Lambda_2\inv \Lambda_{12}(a))\inv
\cdot
\Lambda_1\inv [\Omega(\xi_2,\xi_{34},\Lambda_2\inv \Lambda_{12}(a))]
\\&& \cdot
\Lambda_1\inv[\Omega(\xi_{23},\xi_4,\Lambda_{23}\inv \Lambda_4\inv \Lambda_{34}\Lambda_{12}(a))\inv]
\cdot
\Lambda_1\inv[\Omega(\xi_2,\xi_3,\Lambda_2\inv \Lambda_3\inv \Lambda_4\inv \Lambda_{34}\Lambda_{12}(a))\inv]
\\&&\cdot \Omega(\xi_1,\xi_2,\Lambda_1\inv \Lambda_2\inv \Lambda_3\inv \Lambda_4\inv \Lambda_{34}\inv \Lambda_{12}(a)) \cdot \Lambda_{12}\inv[\Omega(\xi_3,\xi_4, \Lambda_3\inv \Lambda_4\inv \Lambda_{34}\Lambda_{12}(a))]\\
&=&
\Omega(\xi_1,\xi_2,\Lambda_1\inv \Lambda_2\inv \Lambda_{12}(a))\inv
\cdot
\Lambda_1\inv \Big[\Omega(\xi_2,\xi_{34},\Lambda_2\inv \Lambda_{12}(a))
\\& &\cdot
\Omega(\xi_{23},\xi_4,\Lambda_{23}\inv \Lambda_4\inv \Lambda_{34}\Lambda_{12}(a))\inv
\cdot
\Omega(\xi_2,\xi_3,\Lambda_2\inv \Lambda_3\inv \Lambda_4\inv \Lambda_{34} \Lambda_{12}(a))\inv\Big]
\\&&\cdot
\Omega(\xi_1,\xi_2,\Lambda_1\inv \Lambda_2\inv \Lambda_3\inv \Lambda_4\inv \Lambda_{34}\inv \Lambda_{12}(a))
\cdot \Lambda_{12}\inv[\Omega(\xi_3,\xi_4, \Lambda_3\inv \Lambda_4\inv \Lambda_{34} \Lambda_{12}(a))]\\
&\stackrel{\eqref{E 3.3}}{=}&
\Lambda_{12}\inv \Lambda_2\Big[\Omega(\xi_2,\xi_{34},\Lambda_2\inv \Lambda_{12}(a))
\cdot
\Omega(\xi_{23},\xi_4,\Lambda_{23}\inv \Lambda_4\inv \Lambda_{34} \Lambda_{12}(a))\inv
\\&&\cdot
\Omega(\xi_2,\xi_3,\Lambda_2\inv \Lambda_3\inv \Lambda_4\inv \Lambda_{34} \Lambda_{12}(a))\inv\Big]
\cdot \Lambda_{12}\inv[\Omega(\xi_3,\xi_4, \Lambda_3\inv \Lambda_4\inv \Lambda_{34} \Lambda_{12}(a))]\\
&=& \Lambda_{12}\inv \Lambda_2\Big[\Omega(\xi_2,\xi_{34},\Lambda_2\inv \Lambda_{12}(a)) \cdot
\Omega(\xi_{23},\xi_4,\Lambda_{23}\inv \Lambda_4\inv \Lambda_{34} \Lambda_{12}(a))\inv
\\&&\cdot
\Omega(\xi_2,\xi_3,\Lambda_2\inv \Lambda_3\inv \Lambda_4\inv \Lambda_{34} \Lambda_{12}(a))\inv \cdot \Lambda_2\inv[\Omega(\xi_3,\xi_4, \Lambda_3\inv \Lambda_4\inv \Lambda_{34} \Lambda_{12}(a))]\Big]\\
&=&\Lambda_{12}\inv \Lambda_2\Big(\Big[\Xi_3(\xi_2,\xi_3,\xi_4)\big(\Lambda_2\inv \Lambda_{12}(a)\big)\Big]\inv\Big) \\
&=&\Big[\Lambda_{12}\inv \Lambda_2 \Big(\Xi(\xi_2,\xi_3,\xi_4)\big(\Lambda_2\inv \Lambda_{12}(a)\big)\Big)\Big]\inv \\
&=&\Big[\left(\Lambda_{12}\inv(\Lambda_2( \Xi(\xi_2,\xi_3,\xi_4)))\right)(a)\Big]\inv \\
&=&\Big[\Lambda_1\inv( \Xi(\xi_2,\xi_3,\xi_4))(a)\Big]\inv,
\end{eqnarray*}
where for the last equality we have used the fact that the induced action of $\K$ on $Z_\A$ is a morphism from $K^1$ to $\mathrm{Aut}(Z_\A)$. Therefore for every $a\in A^0$,
\begin{eqnarray*}
d\, \Xi(\xi_1,\xi_2,\xi_3,\xi_4)(a)
&=&\Lambda_1\inv (\Xi(\xi_2,\xi_3,\xi_4))(a)\cdot \Theta\\
&=&\Lambda_1\inv (\Xi(\xi_2,\xi_3,\xi_4))(a)\cdot [\Lambda_1\inv (\Xi(\xi_2,\xi_3,\xi_4))(a)]\inv\\
&=&1_a.
\end{eqnarray*}
This shows that $\Xi$ is a 3-cocycle, hence represents a class $[\Xi]$ in $H^3_{\bar \Lambda}(\K,Z_\A)$. This finishes the proof of Theorem \ref{T Apendix-Xi}.

Finally we show that
\begin{theorem}
The cohomology class $[\Xi]$ depends only on $\bar \Lambda$, not on $(\Lambda,\Omega)$.
\end{theorem}
\begin{proof}
First consider the case that, there is another $\Omega'$ satisfying
$
\Omega'(\xi,\eta,\cdot):{\sf id_A}\Rto \Lambda_{\xi\eta}\inv \Lambda_\eta \Lambda_{\xi}.
$
Then we see that
$
\Omega(\xi,\eta,\cdot)\odot \Omega'(\xi,\eta,\cdot)^{\odot,-1}:\sf id_A\Rto id_A.
$
Hence $\Omega(\xi,\eta,\cdot)\odot \Omega'(\xi,\eta,\cdot)^{\odot,-1}\in Z_\A$ by Proposition \ref{P SAut}. Set
$
\rho(\xi,\eta)(a)
=\Omega(\xi,\eta,a)\cdot \Omega'(\xi,\eta,a)\inv.
$
Then $\rho\in C^2_{\bar \Lambda}(\K,Z_\A)$. Denote by $\Xi'$ the cocycle determined by $(\Lambda,\Omega')$ via \eqref{E def-Xi}. Then we find
$
\Xi=d\rho\cdot \Xi'.
$
So $[\Xi]=[\Xi']$.

We next consider the case that there is another lifting $\Lambda''$ of $\bar \Lambda$ and a corresponding $\Omega''$. Then there are a smooth family of natural transformations
\begin{align}\label{E rho-Appendix}
\rho_\xi:\Lambda_\xi\Rto \Lambda''_\xi.
\end{align}
We define
\[
\Omega'(\xi,\eta,\cdot) :=\Omega(\xi,\eta,\cdot)\odot (\rho_{\xi\eta}^{\circledast,-1}\circledast \rho_\eta\circledast\rho_\xi):{\sf id_A}\Rto (\Lambda_{\xi\eta}'')\inv \Lambda''_\eta \Lambda''_\xi.
\]
Denote by $\Xi$ the cocycle determined by $(\Lambda,\Omega)$, by $\Xi'$ the cocycle determined by $(\Lambda'',\Omega')$, and by $\Xi''$ the cocycle determined by $(\Lambda'',\Omega'')$. Then by the argument for first case we have
$
[\Xi']=[\Xi''].
$
On the other hand, the $\rho$ in \eqref{E rho-Appendix} gives rise to the conjugation  transformation from $(\Lambda,\Omega)$ to $(\Lambda'',\Omega')$. Then since $\Xi$ and $\Xi'$ are both sections of $Z_\A$, we have
$
\Xi'=\Xi.
$
Therefore $[\Xi]=[\Xi'']$.
\end{proof}



%


\end{document}